\documentclass[leqno,11pt]{article}
\usepackage{a4wide,amsmath,amssymb,amsfonts,mathrsfs,exscale,graphics,amsthm,colonequals,comment,ifthen,bbold,url,breakurl,enumerate}

\usepackage[curve,matrix,arrow,cmtip]{xy}
\newdir^{ (}{{}*!/-3pt/\dir^{(}}   
\newdir_{ (}{{}*!/-3pt/\dir_{(}}   
\NoComputerModernTips

\numberwithin{equation}{section}

\theoremstyle{plain}
\newtheorem{theorem}[equation]{Theorem}

\newtheorem{lemma}[equation]{Lemma}
\newtheorem{corollary}[equation]{Corollary}
\newtheorem{proposition}[equation]{Proposition}

\theoremstyle{definition}
\newtheorem{definition}[equation]{Definition}
\newtheorem{conjecture}[equation]{Conjecture}
\newtheorem{construction}[equation]{Construction}
\newtheorem{remark}[equation]{Remark}

 \DeclareMathOperator{\Spec}{Spec}
 \DeclareMathOperator{\Spf}{Spf}
 \DeclareMathOperator{\Hom}{Hom}

 \DeclareMathOperator{\Aut}{Aut}

\DeclareMathOperator{\Stab}{Stab}
\DeclareMathOperator{\Frob}{Frob}
\DeclareMathOperator{\F}{F}

\DeclareMathOperator{\Trans}{Trans}
\DeclareMathOperator{\Tr}{Tr}
\DeclareMathOperator{\Img}{Im}

\let\into\hookrightarrow

\newcommand{\defeq}{\colonequals}

\newcommand{\Gm}[1][\empty]{
  \ifthenelse{\equal{#1}{\empty}}
    {\mathbb{G}_m}
    {\mathbb{G}_{m,#1}}}

 \newcommand{\Gred}[1][\empty]{
  \ifthenelse{\equal{#1}{\empty}}
    {G^{\text{red}}}
    {G^{\text{red},#1}}}

 \newcommand{\Rep}[1][\empty]{
  \ifthenelse{\equal{#1}{\empty}}
    {\mathop{\text{\tt Rep}}\nolimits}
    {\mathop{\text{$#1$-{\tt Rep}}}\nolimits}}
\newcommand{\per}{\text{per}}

\DeclareMathOperator{\AFS}{AFS}
\DeclareMathOperator{\Max}{Max}
\DeclareMathOperator{\Tor}{Tor}

\newcommand\lto{\longrightarrow}
\newcommand\ltoover[1]{\mathrel{\smash{\overset{#1}{\lto}}}}

\newcommand\toover[1]{\mathrel{\smash{\overset{#1}{\to}}}}
\newcommand\varto[1]{\mathrel{\hbox to #1pt{\rightarrowfill}}}

\renewcommand{\implies}{\Rightarrow}
\renewcommand{\iff}{\Leftrightarrow}

\let\longto\longrightarrow

\def\isoto{\stackrel{\sim}{\longto}}
 
\newcommand{\BF}{{\mathbb{F}}}

\newcommand{\BQ}{{\mathbb{Q}}}

\newcommand{\BZ}{{\mathbb{Z}}}

\newcommand{\Fm}{{\mathfrak{m}}}
\newcommand{\Fn}{{\mathfrak{n}}}

\newcommand{\Fp}{{\mathfrak{p}}}
\newcommand{\Fq}{{\mathfrak{q}}}

\newcommand{\CA}{{\mathcal A}}
\newcommand{\CB}{{\mathcal B}}

\newcommand{\CI}{{\mathcal I}}

\newcommand{\CO}{{\mathcal O}}

\newcommand{\CT}{{\mathcal T}}

\newcommand{\CX}{{\mathcal X}}
\newcommand{\CY}{{\mathcal Y}}

\newcommand{\Gscr}{{\mathscr G}}
\newcommand{\Hscr}{{\mathscr H}}

\newcommand{\Xscr}{{\mathscr X}}
\newcommand{\Yscr}{{\mathscr Y}}

\let\phi\varphi

\newcommand{\Rper}{{R^\text{per}}}

\begin{document}
\title{Mordell-Lang in positive characteristic}

\author{Paul Ziegler\footnote{Dept. of Mathematics,
 ETH Z\"urich
 CH-8092 Z\"urich,
 Switzerland,
 {\tt paul.ziegler@math.ethz.ch}
\newline
Supported by the Swiss National Science Foundation}
}

\maketitle
\abstract{We give a new proof of the Mordell-Lang conjecture in positive characteristic for finitely generated subgroups. We also make some progress towards the full Mordell-Lang conjecture in positive characteristic.}

\section{Introduction}
\label{sec:introduction}

This article gives an algebro-geometric proof of the following result:
\begin{theorem}[{Hrushovski, \cite{HrushovskiML}}] \label{IntroML}
  Let $L$ be an algebraically closed field of characteristic $p > 0$. Let $A$ be a semiabelian variety over $L$, let $X\subset A$ be an irreducible subvariety and let $\Gamma\subset A(L)$ be a finitely generated subgroup. If $X(L)\cap \Gamma$ is Zariski dense in $X$, then there exist a semiabelian variety $B$ over $\bar \BF_p$, a subvariety $Y$ of $B$ over $\bar \BF_p$, a homomorphism $h\colon B_L\to A/\Stab_A(X)$ with finite kernel and an element $a\in (A/\Stab_A(X))(L)$ such that $X/\Stab_A(X)=h(Y)+a$.
\end{theorem}
Here $\Stab_A(X)$ denotes the translation stabilizer of $X$ in $A$. We call irreducible subvarieties $X$ satisfying the conclusion of the above theorem \emph{special}. 

An algebro-geometric proof of Theorem \ref{IntroML} has previously been given by R\"ossler \cite{Roessler1}.

We give an overview of the structure of this article. In Section \ref{sec:FormalSchemes} we collect a number of facts about a class of formal schemes which includes the closed formal subschemes of the completion of a semiabelian scheme over a discrete valuation ring along its zero section. In Section \ref{sec:SpecialSubvarieties} we collect some facts about special subvarieties and give a criterion for special subvarieties (Theorem \ref{InfSpecialCrit}) based on the classification of subvarieties of semiabelian subvarieties which are invariant under an isogeny due to Pink and R\"ossler. 

In Section \ref{sec:setup} we set up our method for proving Theorem \ref{IntroML}. In Subsection \ref{sec:CSD} we collect some facts about certain $p$-divisible groups. These are $p$-divisible groups which possess a filtration such that on each graded piece, a power of the relative Frobenius coincides, up to an isomorphism, with a power of the multiplication-by-$p$ morphism. 

In Subsection \ref{sec:NiceSetup}, we construct our central tool, a certain Frobenius morphism $F$: Let $R$ be the valuation ring of a local field $K$ of positive characteristic, let $k$ be the residue field of $R$, let $\bar K$ be an algebraic closure of $K$ and let $\Rper\subset \bar K$ be the perfection of $R$. Let $\CA$ be a semiabelian scheme over $R$. Denote by $\hat\CA$ the completion of $\CA$ along the zero section of the special fiber. This is a $p$-divisible group over $\Spf(R)$. Assume that $\hat\CA$ is completely slope divisible. Then the facts from Subsection \ref{sec:CSD} yield a canonical isomorphism $(\hat\CA_k)_{\Spf(\Rper)} \cong \hat\CA_{\Spf(\Rper)}$. By transfering the Frobenius endomorphism of $\hat\CA_k$ with respect to $k$ via this isomorphism, we obtain an endomorphism $F$ of $\hat\CA_\Rper$. Its significance lies in the following characterization of special subvarieties of $\CA_{\bar K}$, where for a subvariety $\CX$ of $\CA$ containing the zero section we denote by $\hat \CX$ its completion along the zero section:
\begin{theorem}[see Theorem \ref{SpecialFChar}]\label{SpecialFCharIntro}
     Let $\CX\subset \CA$ be an irreducible subvariety. Then the following are equivalent:
  \begin{enumerate}[(i)]
  \item The subvariety $\CX_{\bar K}$ is special in $\CA_{\bar K}$.
  \item There exist $x\in \CX(\bar K)$ and $n\geq 1$ such that $F^n_{\hat\CA}(\widehat{X-x})\subset\widehat{X-x}$.
  \end{enumerate}

\end{theorem}
This can be considered as a formal analogue of the classification of subvarieties of semiabelian subvarieties which are invariant under an isogeny due to Pink and R\"ossler. Finally, in Subsection \ref{sec:NiceValuation}, we show that given a finitely generated field $L_0$ and a semiabelian variety $A$ over $L_0$, up to isogeny one can always spread out $A$ to an abelian scheme $\CA$ as above.

In Section \ref{sec:proof1}, using the methods from Section \ref{sec:setup}, we prove Theorem \ref{IntroML}: First we reduce to the situation considered above which allows us to define $F$. Then using Theorem \ref{SpecialFCharIntro} we show that Theorem \ref{IntroML} follows from a variant of the following formal Mordell-Lang result:
\begin{theorem}[see Theorem \ref{FormalML}]  \label{IntroFormalML}
  Let $\Gscr$ be a formal group over $k$ which as a formal scheme is isomorphic to $\Spf(k[[x_1,\hdots,x_n]])$. Let $\Gamma\subset \Gscr(R)$ a finitely generated subgroup. Let $\Xscr \subset \Gscr_R$ be a closed formal subscheme. If $\Xscr$ is the minimal closed formal subscheme containing $\Xscr(R) \cap \Gamma$, then there exist closed formal subschemes $\Xscr_1,\hdots,\Xscr_m$ of $\Gscr_R$ defined over a finite field extension of $k$ and elements $\gamma_1,\hdots,\gamma_m \in\Gscr(R)$ such that $\Xscr=\cup_j \Xscr_j+\gamma_j$.
\end{theorem}
Theorem \ref{IntroFormalML} is proven by the same method which was used by Abramovich and Voloch in \cite{AbramovichVoloch} to prove Theorem \ref{IntroML} in the isotrivial case: First one reduces to the case that $\Xscr$ is irreducible in a suitable sense. Then after a suitable translation one may assume that $\Xscr(R)\cap p^i\Gamma$ is dense is $\Xscr$ for all $i$. Denote by $R^{p^i}$ the subring of $R$ consisting of all $p^i$-th powers. Since $p^i\Gamma\subset \Gscr(R^{p^i})$ it follows that $\Xscr$ is defined over $R^{p^i}$ for all $i\geq 0$. Hence $\Xscr$ is defined over $k=\cap_{i\geq 0}R^{p^i}$.

In Section 5 we consider the full Mordell-Lang conjecture, which is the statement obtained by allowing the group $\Gamma$ in Theorem \ref{IntroML} to be of finite rank. This conjecture is still open in general. We show that in case $A$ is ordinary by combining our method with a reduction due to Ghioca, Moosa and Scanlon, Conjecture \ref{fullML} can be reduced to the following special case:
\begin{conjecture}\label{fullMLreduction2Intro}
  Let $L_0$ be a field which is finitely generated over $\BF_p$, let $L$ an algebraic closure of $L_0$ and let $L_0^\per$ be the perfect closure of $L_0$ in $L$. Let $A$ be a semiabelian variety over $L_0$ and $X \subset A_{L_0^\per}$ an irreducible subvariety. Assume that the canonical morphism $\Tr_{L/\bar \BF_p}A\to A$ is defined over $L_0$, that there exists a finite subfield $\BF_q$ of $L_0$ over which $\Tr_{L/\bar\BF_p}A$ can be defined and that $\Stab_{A_{L_0^\text{per}}}(X)$ is finite. If $X(L_0^\per)$ is Zariski dense in $X_{0}$, then a translate of $X$ by an element of $A(L_0^\text{per})$ is defined over $L_0$.
\end{conjecture}
This depends crucially on the fact that in case $A$ is ordinary, the endomorphism $F$ of $\hat\CA_{\Spf(\Rper)}$ described above can already be defined over $R$. The argument proceeds similarly to the proof of Theorem \ref{IntroML} sketched above by reduction to an analogous statement (see Theorem \ref{FormalDescent}) for formal group schemes.

\paragraph{Acknowledgement} I am deeply grateful to Richard Pink for suggesting this topic to me and for his guidance. I thank Damian R\"ossler and Thomas Scanlon for pointing out a mistake in an earlier version of this article and for helpful conversations. I also thank Ambrus P\'al for a helpful conversation.

\section{Formal Schemes} \label{sec:FormalSchemes}
Let $R$ be a the valuation ring of a local field $K$ of characteristic $p>0$. Denote by $\Fm$ the maximal ideal of $R$ and by $k$ the residue field of $R$. Let $\bar K$ be an algebraic closure, let $\bar R$ be the valuation ring of the unique extension of the valuation of $R$ to $\bar K$ and let $\bar\Fm$ be the maximal ideal of $\bar R$. 

We denote by $\hat{\bar K}$ the completion of $\bar K$. By a complete overfield $K'\subset \hat{\bar K}$ of $K$ we mean a field which is complete with respect to the valuation induced from $\hat{\bar K}$. The valuation ring of such a $K'$ will be denoted $R'$ and the the formal scheme associated to $R'$ with its equipped with the valuation topology will be denoted by $\Spf(R')$.

 By an adic ring we mean the same as in \cite{EGA1}, that is a complete and separated topological ring whose topology is defined by an ideal $J$. We will also call such a ring a $J$-adic ring.

For $R'$ as above and $n,m\geq 0$ we denote by $C'_{n,m}$ the ring $R'[[x_1,\hdots,x_n]]\langle y_1,\hdots,y_m\rangle$ which consists of those power series $\sum_{I\in (\BZ^{\geq 0})^n, J\in(\BZ^{\geq 0})^m}a_{IJ}x^Iy^J$ with coefficients from $R'$ such that for each $I$ the coefficients $a_{IJ}$ converge to zero as $J$ goes to infinity. We endow $C'_{n,m}$ with the topology defined by the ideal $J'_{n,m}$ generated by $\Fm$ and the variables $x_1,\hdots,x_n$. This makes $C'_{n,m}$ into an adic ring. For $R'=R$ we let $C_{n,m}\defeq C'_{n,m}$ and $J_{n,m}\defeq J'_{n,m}$.

 By formal schemes, we mean the same as in \cite[Section 10]{EGA1}. In this section, we are concerned with affine formal schemes $\Xscr$ over $\Spf(R)$ defined by the following class of rings:

 \begin{definition}[{c.f. \cite[Section 2.1]{KappenUniform} and \cite[Section 1]{BerkovichVanishingCyclesII}}]
   A topological $R$-algebra $C$ is of \emph{formally finite type} if it is adic and if for some ideal of definition $J$ the quotients $C/J^i$ are of finite type over $R$ for all $i\geq 0$.
 \end{definition}
\begin{definition}
  We denote by $\AFS_R$ the full subcategory of of the category of formal schemes over $\Spf(R)$ whose objects are the formal schemes of the form $\Spf(C)$ for $C$ a topological $R$-algebra of formally finite type. 
\end{definition}

\begin{lemma}[{\cite[Lemma 1.2]{BerkovichVanishingCyclesII}}] \label{FFChar}
   For a $J$-adic $R$-algebra $C$ the following are equivalent:
   \begin{enumerate}[(i)]
   \item The ring $C$ is of formally finite type.
   \item The ring $C/J^2$ is finitely generated over $R$.
   \item The ring $C$ is topologically isomorphic over $R$ to a quotient of $C_{n,m}$ for some $n,m\geq 0$.
   \end{enumerate}
 \end{lemma}

\begin{remark} 
  Let $\Xscr=\Spf(C)\in\AFS_R$. By the remark after \cite[Definition 10.14.2]{EGA1} closed formal subschemes of $\Xscr$ correspond to ideals of $C$. Thus by Lemma \ref{FFChar} a formal scheme over $\Spf(R)$ is in $\AFS_R$ if and only if it admits a closed embedding into $\Spf(C_{n,m})$ for some $n,m\geq 0$. 
\end{remark}

The following summarizes properties of topological $R$-algebras of formally finite type:
\begin{proposition} \label{CProps}
  Let $C$ and $C'$ be topological $R$-algebras of formally finite type.
  \begin{enumerate}[(i)]
  \item The Jacobson radical of $C$ is an ideal of definition, in fact it is the largest ideal of definition. In particular there is a unique topology on the ring $C$ which makes $C$ into a topological $R$-algebra of formally finite type.
  \item Every homomorphism $C\to C'$ is continuous.
  \item If $C\to C'$ is a surjection and $I$ an ideal of definition of $C$, then the ideal generated by the image of $I$ is an ideal of definition of $C'$.
  \item Each ideal of $C$ is closed.
  \item The ring $C\otimes_R K$ is Jacobson.
  \item For each maximal ideal $\Fn$ of $C\otimes_R K$ the quotient $(C\otimes_R K)/\Fn$ is a finite field extension of $K$.
  \end{enumerate}
\end{proposition}
\begin{proof}
 For $(i)$ and $(ii)$ see \cite[Lemma 2.1]{KappenUniform}. For $(iii)$ and $(iv)$ see \cite[Lemma 1.1]{BerkovichVanishingCyclesII}. For $(v)$ ee \cite[Proposition 2.16]{KappenUniform} and for $(vi)$ see \cite[Lemma 2.3]{KappenUniform}.
\end{proof}

We will also have to work with formal schemes $\Xscr_{\Spf(R')}$ for $\Xscr\in \AFS_R$ and $R'$ the valuation ring of a complete overfield $K'\subset \hat{\bar K}$ and with closed formal subschemes of such formal schemes. However, in \cite{EGA1} the notion of a formal subscheme is only defined for locally Noetherian formal schemes, and valuation rings $R'$ as above are in general not Noetherian. Thus we make the following definition:
\begin{definition}
  A morphism $\Spf(C)\to\Spf(C')$ of affine formal schemes is a \emph{closed embedding} if the corresponding homomorphism $C'\to C$ is surjective and the topology on $C$ is the quotient topology induced from $C'$. In this case we will say that $\Spf(C)$ is a \emph{closed formal subscheme} of $\Spf(C')$. 
\end{definition}
Thus closed formal subschemes of $\Spf(C')$ correspond to closed ideals of $C'$. In case $C'$ is Noetherian, this definition coincides with the one from \cite{EGA1} by the remark after \cite[Definition 10.14.2]{EGA1}. 

\begin{definition}
  Let $R'$ be the valuation ring of a complete overfield $K'\subset \hat{\bar K}$. We denote by $\AFS_{R'}$ the full subcategory of the category of formal schemes over $\Spf(R')$ whose objects are those formal schemes which admit a closed embedding into $\Spf(C'_{n,m})$ for some $n,m\geq 0$. 
\end{definition}

\begin{lemma} \label{BaseChangeExact}
  Let $C$ be a Noetherian $J$-adic ring and $C'$ a $J'$-adic ring. Let $C\to C'$ be a ring homomorphism such that $JC'\subset J'$ and such that for each $i\geq 0$ the induced homomorphism $C/J^i\to C'/(J')^i$ is faithfully flat.

Let $0\to M' \to M\to M''\to 0$ a sequence of finitely generated $C$-modules. We endow these modules with the $J$-adic topology.  Then the sequence $0\to M'\to M\to M''\to 0$ is exact if and only if $0\to M'\hat\otimes_C C'\to M\hat\otimes_C C'\to M''\hat\otimes_C C'\to 0$ is exact.
\end{lemma}
\begin{proof}
    Since $JC'\subset J'$ the completed tensor product $M\hat\otimes_C C'$ can be written as
  \begin{equation*}
    M\hat\otimes_C C'=\varprojlim_i M\otimes_C C'/(J')^i,
  \end{equation*}
and analogously for $M'$ and $M''$. 

Assume that $0\to M'\to M\to M''\to 0$ is exact.

 For $i\geq 0$ let $M'_i\defeq J^iM\cap M'$. By \cite[Theorem III.3.2.2]{BourbakiCA} the topology on $M'$ defined by the $M'_i$ is the $J$-adic topology. This together with the fact that $JC'\subset J'$ implies that $(M_i\otimes_C (J')^iC')_{i\geq 0}$ is a fundamental system of neighborhoods of the identity in $M'\otimes_C C'$. Thus $M'\hat\otimes_C C'$ can be written as
 \begin{equation*}
     M'\hat\otimes_C C'=\varprojlim_i M'/M'_i\otimes_C C'/(J')^i.
 \end{equation*}

 For $i\geq 0$ there is an exact sequence $0\to M'/M'_i\to M/J^iM\to M''/J^i M''\to 0$ of $C/J^i$-modules. Since $C/J^i\to C'/(J')^i$ is flat, this induces an exact sequence 
 \begin{equation*}
   0\to M'/M'_i\otimes_{C/J^i} C'/(J')^i\to M/J^iM\otimes_{C/J^i} C'/(J')^i \to M''/J^i M''\otimes_{C/J^i} C'/(J')^i \to 0.
 \end{equation*}
This sequence can also be written as
\begin{equation*}
   0\to M'/M'_i\otimes_{C} C'/(J')^i\to M\otimes_{C} C'/(J')^i \to M''\otimes_{C} C'/(J')^i \to 0.
\end{equation*}
 The transition morphisms $M'/M'_i\otimes_C C'/(J')^i \to M'/M'_{i-1}\otimes_C C'/(J')^{i-1}$ are surjective. By \cite[Proposition 10.2]{AtiyahMacdonald} this surjectivity implies that the sequence $0\to M'\hat\otimes_C C'\to M\hat\otimes_C C'\to M''\hat\otimes_C C'\to 0$ obtained by taking the inverse limit of the above exact sequences is again exact. 

To prove the other direction of the claim, by a direct verification it suffices to show that if $M$ is a non-zero finitely generated $C$-module endowed with the $J$-adic topology, the ring $M\hat\otimes_C C'$ is non-zero. As above we can write $M \hat\otimes_C C'$ as $\varprojlim_i M/J^iM \otimes_{C/J^i} C'/(J')^i$ with surjective transition homomorphisms $M/J^iM \otimes_{C/J^i} C'/(J')^i\to M/J^{i-1}M\otimes_{C/J^{i-1}} C'/(J')^{i-1}$. As it is finitely generated over the complete Noetherian ring $C$, the module $M$ is complete. Hence the modules $M/J^iM$ are non-zero. Thus by the faithful flatness of $C/J^i\to C'/(J')^i$ the modules $M/J^iM\otimes_{C/J^i}C'/(J')^i$ are non-zero. Hence the surjectivity of the transition morphisms implies that $M\hat\otimes_C C'$ is non-zero.
\end{proof}

\begin{lemma} \label{BaseChangeFlat}
  Let $R'$ be the valuation ring of a completely valued overfield $K'\subset \hat{\bar K}$ of $K$. For $n,m,i \geq 0$ the ring homomorphism $C_{n,m}/J^i_{n,m}\to C'_{n,m}/(J'_{n,m})^i$ induced by the inclusion $C_{n,m}\into C'_{n,m}$ is faithfully flat.
\end{lemma}
\begin{proof}
  The homomorphism in question is
  \begin{align*}
     R/\Fm^i[x_1,\hdots,x_n,y_1,\hdots,y_n]/(x_1,\hdots,x_n)^i&\to R'/(\Fm R')^i[x_1,\hdots,x_n,y_1,\hdots,y_n]/(x_1,\hdots,x_n)^i \\
     &\cong R'\otimes_R R/\Fm^i[x_1,\hdots,x_n,y_1,\hdots,y_n]/(x_1,\hdots,x_n)^i.
  \end{align*}
Thus the claim follows from the faithful flatness of $R\to R'$.
\end{proof}
\begin{lemma} \label{ClosedSubschemeBaseChange}
  Let $\Xscr=\Spf(C)\in \AFS_R$ and $\Xscr'=\Spf(C')$ a closed formal subscheme of $\Xscr$ defined by an ideal $I$ of $C$. Let $R'$ be the valuation ring of a complete overfield $K'\subset \hat{\bar K}$ of $K$. Then $\Xscr'_{\Spf(R')}$ is the closed formal subscheme of $\Xscr_{\Spf(R')}$ defined by the ideal $I(C\hat\otimes_R R')$ of $C\hat\otimes_R R'$. This ideal is equal to $I\hat\otimes_RR'$.
\end{lemma}
\begin{proof}
 Pick a surjection $C_{n,m}\to C$ for some $n,m\geq 0$. Note that $M\hat\otimes_R R'=M\hat\otimes_{C_{n,m}}C'_{n,m}$ for any topological $C$-module $M$.   By Proposition \ref{CProps} the topology on $C'$ is the same as the topology defined by $J_{n,m}C'$. Hence by Lemma \ref{BaseChangeExact} applied to $C_{n,m}\to C'_{n,m}$, which is possible by Lemma \ref{BaseChangeFlat}, there is an exact sequence $0\to I\hat\otimes_R R'\to C\hat\otimes_R R'\to (C/I)\hat\otimes_R R'\to 0$. Thus $\Xscr'$ is the closed formal subscheme of $\Xscr_{\Spf(R')}$ defined by the ideal $I\hat\otimes_R R'$ in $C\hat\otimes_R R'$. Since $C$ is Noetherian, there is a surjective homomorphism of $C$-modules $C^{\oplus k}\to I$ for some $k\geq 0$. Again by Lemma \ref{BaseChangeExact} this induces a surjection $(C\hat\otimes_R R')^{\oplus k}\cong C^{\oplus k}\hat\otimes_R R'\to I\hat\otimes_R R'$ which implies $I\hat\otimes_R R'=I(C\hat\otimes_R R')$. 
\end{proof}
\begin{definition} \label{UnionDef}
  Let $\Xscr=\Spf(C)$ be an affine formal scheme and $\Xscr_1,\hdots,\Xscr_m$ be closed formal subschemes of $\Xscr$ defined by closed ideals $I_1,\hdots,I_m$ of $C$. We say that $\Xscr$ is the union of the formal subschemes $\Xscr_i$ if the intersection of the ideals $I_i$ is the zero ideal of $C$.
\end{definition}

 \subsection{Points over $\bar R$}
\begin{definition}
  Let $\Xscr=\Spf(C)\in \AFS_R$. We define $\Xscr(\bar R)$ to be the set of homomorphisms $C \to \bar R$ of $R$-algebras. 
\end{definition}

\begin{lemma} \label{hProp}
  Let $C$ be a topological $R$-algebra of formally finite type and let $h\colon C\to\bar R$ be a homomorphism of $R$-algebras.
  \begin{enumerate}[(i)]
  \item The homomorphism $h$ factors through the valuation ring $R'$ of a finite field extension $K'\subset \bar K$ of $K$.
  \item The homomorphism $h$ is continous.
  \end{enumerate}
\end{lemma}
\begin{proof}

$(i)$: The homomorphism $h$ induces a homomorphism $C\otimes_R K\to \bar R\otimes_R K\to \bar K$ with the last homomorphism given by multiplication. Its kernel is a maximal ideal $\Fn$ of $C\otimes_R K$. By Proposition \ref{CProps} the quotient $(C\otimes_RK)/\Fn$ is a finite field extension of $K$. This implies $(i)$.

$(ii)$: By $(i)$ it suffices to show that if $R'$ is the valuation ring of a finite field extension $K'\subset \bar K$ of $K$ any homomorphism $h\colon C \to R'$ of $R$-algebras is continuous. Since $R'$ is of formally finite type this is a special case of Proposition \ref{CProps} (ii).
\end{proof}

 Caution: The ring $\bar R$ is not complete with respect to the valuation topology. Thus there is no formal scheme $\Spf(\bar R)$ and $\Xscr(\bar R)$ cannot be considered as $\Xscr(\Spf(\bar R))$. The set $\Xscr(\bar R)$ is also not the same as $\Xscr(\Spf(\hat{\bar R}))$.

Our interest in the set $\Xscr(\bar R)$ comes from the following situation, to which we will later apply the results of this section: Let $\CA$ be a semiabelian scheme of dimension $g$ over $R$ and $\CX\subset \CA$ a closed subscheme containing the zero section. Let $s$ be the closed point of the zero section of $\CA$. Let $\hat\CA$ and $\hat\CX$ be the formal schemes associated to the completions of the local rings $\CO_{\CA,s}$ and $\CO_{\CX,s}$ with respect to their maximal ideals. Then $\hat\CA\cong \Spf(R[[x_1,\hdots,x_g]])$ so that $\hat\Xscr\in\AFS_R$ and $\hat\CX(\bar R)$ is naturally identified with the set points in $\CX(\bar R)$ which map to $0$ in the special fiber (c.f. Subsection \ref{FormalSchemeFromScheme}).

Note that the formation of $\Xscr(\bar R)$ is functorial in $\Xscr$.

For a finite field extension $K'\subset \bar K$ of $K$ with valuation ring $R'$ and $\Xscr\in\AFS_R$ we denote by $\Xscr(R')$ the set of homomorphisms $\Xscr\to\Spf(R')$ over $\Spf(R)$. There is a natural inclusion $\Xscr(R')\into \Xscr(\bar R)$ and Lemma \ref{hProp} $(i)$ implies:
\begin{lemma} \label{XRBarUnion}
  Let $\Xscr\in\AFS_R$. The set $\Xscr(\bar R)$ is the union $\Xscr(\bar R)=\cup_{K'}\Xscr(R')$ where $K'$ varies over all finite field extensions of $K$ contained in $\bar K$.
\end{lemma}
\begin{remark} \label{XBarExplicit}
For $\Xscr\in\AFS_R$, the set $\Xscr(\bar R)$ can be described more concretely as follows: For any $r_1,\hdots ,r_n\in\bar\Fm$ and $s_1,\hdots,s_n\in \bar R$ there exists a unique continous homomorphism of $R$-algebras $C_{n,m}\to \bar R$ which sends the $x_i$ to $r_i$ and the $y_i$ to $s_i$ and each homomorphism $C_{n,m}\to \bar R$ is of this form. Thus associating to an element $h\in \Spf(C_{n,m})(\bar R)$ the images of the $x_i$ and the $y_i$ gives a bijection $\Spf(C_{n,m} )(\bar R) \isoto \bar\Fm^{\oplus n} \oplus \bar R^{\oplus m}$. Any closed subscheme $\Xscr$ of $\Spf(C_{n,m})$ is cut out by a family of formal power series $\{f_i\mid i\in I\}\subset C_{n,m}$. Each formal power series $f\in C_{n,m}$ induces a function $\bar\Fm^{\oplus n}\oplus \bar R^{\oplus m} \to R$. The above bijection identifies $\Xscr(\bar R)$ with the set of points in $\bar\Fm^{\oplus n}\oplus \bar R^{\oplus m}$ on which the $f_i$ are zero.  
\end{remark}
 \begin{definition}
    Let $\Gamma\defeq\Aut_R(\bar R)\cong \Aut_K(\bar K)$. For $\Xscr\in\AFS_R$, we let $\Gamma$ act on $\Xscr(\bar R)$ from the left by 
    \begin{align*}
      \Gamma\times\Xscr(\bar R)&\to \Xscr(\bar R)\\
      (\gamma, h)&\mapsto \gamma\cdot h\colon \Gamma(\Xscr,\CO_\Xscr)\toover{h} \bar R \toover{\gamma} \bar R. 
    \end{align*}

 \end{definition}
For a ring $C$, we denote by $\Max(C)$ the set of maximal ideals of $C$ equipped with the Zariski topology. 
\begin{proposition} \label{AnalyticZariski}
    Let $\Xscr=\Spf(C)\in \AFS_R$. Let $\psi$ be the map $\Xscr(\bar R)\to \Max(C\otimes_RK)$ which associates to $h\in \Xscr(\bar R)$ the kernel of the induced homomorphism $h\otimes_R K\colon C\otimes_RK\to \bar R\otimes_R K\cong \bar K$. 
    \begin{enumerate}[(i)]
    \item The map $\psi$ makes $\Max(C\otimes_R K)$ into the set-theoretic quotient of $\Xscr(\bar R)$ by the action of $\Gamma$.
    \item Let $\Yscr=\Spf(C')$ be a closed formal subscheme of $\Xscr$. Then there is a commutative diagram 
      \begin{equation*}
         \xymatrix{
          \Yscr(\bar R) \ar[r]^{\psi'} \ar[d] & \Max(C'\otimes_R K) \ar[d] \\
          \Xscr(\bar R) \ar[r]^{\psi} & \Max(C\otimes_RK) 
}
      \end{equation*}
in which $\psi'$ is the analogue of $\psi$ for $\Yscr$  and $\Max(C'\otimes_R K)\to \Max(C\otimes_RK)$ is induced by the surjection $C\to C'$.
    \end{enumerate}

    \end{proposition}
    \begin{proof}
      $(i)$: It follows directly from the definition of $\psi$ that $\psi(h\cdot\gamma)=\psi(h)$ for all $\gamma\in\Gamma$ and $h\in\Xscr(\bar R)$. On the other hand let $h$ and $h'$ be two elements of $\Xscr(\bar R)$ which have the same image $\Fn$ under $\psi$. By Proposition \ref{CProps} the quotient $(C\otimes_RK)/\Fn$ is a finite field extension $K'$ of $K$. The homomorphisms $h\otimes_R K$ and $h'\otimes_RK$ correspond to two different embeddings $i,i$ of $K'$ into $\bar K$ over $K$. There exists $\tilde \gamma\in \Aut_K(\bar K)$ such that $i'=\tilde\gamma\circ i$ and hence the restriction $\gamma\in \Gamma$ of $\tilde\gamma$ to $\bar R$ sends $h$ to $h'$. Thus the fibers of $\psi$ are exactly the $\Gamma$-orbits, which shows $(i)$.

$(ii)$ follows by a direct verification. 
    \end{proof}

 \begin{definition}
   Let $R'$ be the valuation ring of a completely valued overfield $K'$ of $K$. Let $\Xscr=\Spf(C)$ be an affine formal scheme over $\Spf(R')$.
   \begin{enumerate}[(i)]
   \item The formal scheme $\Xscr$ is \emph{reduced} if the ring $C$ is reduced.
   \item The formal scheme $\Xscr$ is \emph{flat over $R'$} if the ring $C$ is flat over $R'$.
   \item We denote by $\AFS_R^\text{rf}$ the full subcategory of $\AFS_R$ whose objects are those formal schemes which are reduced and flat over $R$.
   \end{enumerate}
 \end{definition}
 \begin{lemma} \label{Xrf}
Let $\Xscr=\Spf(C)\in\AFS_R$. Let $\Xscr^\text{rf}$ be the closed formal subscheme of $\Xscr$ defined by the ideal $\{c\in C\mid \exists n \geq 0\colon (a\Fm)^n=0\}$. This formal scheme is reduced and flat over $R$ and the natural map $\Xscr^\text{rf}(\bar R)\to \Xscr(\bar R)$ is a bijection.
 \end{lemma}
 \begin{proof}
   Direct verification using the fact that an $R$-module is flat if and only if it has no $\Fm$-torsion.
 \end{proof}

 \begin{proposition} \label{FormalJacobson}
   Let $\Yscr_1, \Yscr_2 \in \AFS_R^\text{rf}$ be two closed formal subschemes of $\Xscr\in \AFS_R$. If $\Yscr_1(\bar R)\subset \Yscr_2(\bar R)$, then $\Yscr_1\subset \Yscr_2$. 
 \end{proposition}
 \begin{proof}
  Let $C$ be the ring of global sections of $\Xscr$ and $I_1,I_2$ the ideals defining $\Yscr_1,\Yscr_2$. We want to prove $I_2\subset I_1$. Let $\pi$ be a uniformizer of $R$. The fact that the $\Yscr_i$ are flat over $R$ means that $\pi$ is not a zero-divisor in $C/I_i$. This implies that it is enough to prove $I_2\otimes_RK\subset I_1\otimes_RK$ inside $C\otimes_RK$. Since by assumption the ideals $I_i$ are radical, so are the ideals $I_i\otimes_R K$. Since by Proposition \ref{CProps} the ring $C\otimes_RK$ is Jacobson it suffices to prove that each maximal ideal of $C\otimes_RK$ which contains $I_1\otimes_R K$ also contains $I_2\otimes_R K$. This follows from the fact that $\Yscr_1(\bar R)\subset \Yscr_2(\bar R)$ and Proposition \ref{AnalyticZariski}.
 \end{proof}
 For $i\geq 0$, we endow the ring $\bar R/(\Fm \bar R)^i$ with the quotient topology induced from the valuation topology on $\bar R$, with respect to which it is adic. Hence there is a formal scheme $\Spf(\bar R/(\Fm \bar R)^i)$ and for $\Xscr\in\AFS_R$ we denote by $\Xscr(\bar R/(\Fm \bar R)^i)$ the set of morphisms $\Spf(\bar R/(\Fm \bar R)^i)\to \Xscr$ over $\Spf(R)$. There is a natural map $\Xscr(\bar R)\to \Xscr(\bar R/(\Fm \bar R)^i)$ for all $i\geq 0$.
 \begin{corollary} \label{FormalSchemeNonempty}
   Let $\Xscr\in\AFS_R$. The set $\Xscr(\bar R)$ is nonempty if and only if for all $i\geq 0$ the set $\Xscr(\bar R/\Fm^i\bar R)^i$ is nonempty.
 \end{corollary}
 \begin{proof}
   The ``only if'' direction is clear. Conversely, assume that $\Xscr(\bar R)$ is empty. Let $\Xscr^\text{rf}\in \AFS_R^\text{rf}$ be the closed formal subscheme given by Lemma \ref{Xrf}. Since $\Xscr^\text{rf}(\bar R)=\Xscr(\bar R)$ is empty, it follows from Proposition \ref{FormalJacobson} that $\Xscr^\text{rf}$ is the empty formal scheme. Hence, if we let $C\defeq \Gamma(\Xscr,\CO_\Xscr)$ and $\pi$ is a uniformizer of $R$, it follows from the definition of $\Xscr^\text{rf}$ that for each $c\in C$ there exists $n \geq 0$ such that $(c\pi)^n=0$. For $c=1$ we get that there is an $n\geq 0$ such that the image $\pi^{n}$ in $C$ is zero. Hence for all $i>n$ there is no homomorphism $C\to \bar R/(\Fm\bar R)^i$ of $R$-algebras since for such $i$ the image of $\pi^n$ in $\bar R/(\Fm \bar R)^i$ is not zero. This proves the claim.
 \end{proof}

\begin{lemma} \label{XBarRUnion}
 Let $\Xscr=\Spf(C)\in\AFS_R$ be be the union of closed subschemes $\Xscr_1,\hdots,\Xscr_m$. Then $\Xscr(\bar R)=\cup_i\Xscr_i(\bar R)$. 
\end{lemma}
\begin{proof}
  Let $I_1,\hdots,I_m\subset C$ be the ideals defining the $\Xscr_i$. Let $h\colon C\to \bar R$. We want to show that $h(I_i)=0$ for some $i$. If this is not the case we pick $0\not= r_i\in h(I_i)$. Then the product $r$ of the $r_i$ is a non-zero element which lies in $h(I_i)$ for all $i$. For each $i$ pick $c_i\in I_i$ such that $h(c_i)=x$. Then the product of the $c_i$ lies in the intersection of the $I_i$ which by assumption is zero. Thus by applying $h$ to this product we get $x^m=0$ and hence $x=0$, which is a contradiction.
\end{proof}
 \subsection{Irreducibility}
\begin{definition} \label{IrredX}
  Let $\Xscr\in\AFS_R^\text{rf}$ be non-empty.
  \begin{itemize}
  \item The formal scheme $\Xscr$ is \emph{irreducible} if and only if $\Spec(C)$ is irreducible.
  \item An irreducible component of $\Xscr$ is a maximal irreducible closed formal subscheme.
  \item We call the formal scheme $\Xscr$ is \emph{geometrically irreducible} if and only if $\Xscr_{R'}$ is irreducible for all valuation rings $R'$ of finite field extensions $K'$ of $K$.
  \end{itemize}
\end{definition}
Note that the irreducible components of $\Xscr$ correspond to the irreducible components of $\Spec(C)$, that is to the minimal prime ideals of $C$. In particular there are finitely many such components. Also, since $C$ is reduced, the intersection of all its minimal prime ideals is the zero ideal. Thus $\Xscr$ is the union of its irreducible components in the language of Definition \ref{UnionDef}. 

 \begin{proposition}
   Let $\Xscr\in\AFS_R^\text{rf}$ be non-empty.
   Each irreducible component of $\Xscr$ is reduced and flat over $R$.
 \end{proposition}
  \begin{proof}
     Let $C\defeq\Gamma(\Xscr,\CO_\Xscr)$. It suffices to show that each irreducible component of $\Spec(C)$ is reduced and flat over $R$. Reducedness is clear. As $R$ is a discrete valuation ring, a scheme $X$ over $\Spec(R)$ is flat if and only if its generic fiber is schematically dense in $X$. By a direct verification, if $\Spec(C)$ satisfies this condition, then so does any irreducible component of $\Spec(C)$.

   \end{proof}
   \begin{lemma} \label{IrrUnion}
     Let $\Xscr=\Spf(C)\in\AFS_R^\text{rf}$ be irreducible. If $\Xscr_1,\hdots,\Xscr_m$ are closed formal subschemes of $\Xscr$ such that $\Xscr$ is the union of the $\Xscr_i$, then $\Xscr=\Xscr_i$ for some $i$.
   \end{lemma}
   \begin{proof}
      Let $I_1,\hdots,I_m\subset C$ be the ideals defining the $\Xscr_i$. By assumption their intersection is zero. By assumption $C$ is integral. If all $I_i$ were non-zero, we could pick elements $0\not=x_i\in I_i$ whose product would be zero. Thus one of the $I_i$ is zero, which is what we wanted.
   \end{proof}
   In \cite[Section 7]{deJong1}, de Jong gives a construction, due to Berthelot, of a ``generic fiber'' functor from $\AFS_R$ to the category of quasi-separated rigid analytic spaces over $K$. We denote this functor by $\Xscr \mapsto \Xscr^\text{rig}$. It can be described as follows: The formal scheme $\Spf(C_{n,m})$ is sent to the product of the open $n$-dimensional unit disc $D^n_K$ over $K$ and the closed $m$-dimensional unit disc $B^m_K$ over $K$. A closed formal subscheme $\Xscr$ as above is cut out by a family of power series $\{f_i\mid i\in I\}\subset C_{n,m}$. These $f_i$ induce global sections of $D^n_K\times B^m_K$, and $\Xscr^\text{rig}$ is the closed rigid analytic subspace of $D^m_K\times B^n_K$ cut out by these global sections. 
 
We will prove Proposition \ref{GeomIrrComp} by using results from \cite{ConradIrreducible}. There Conrad introduces the notion of irreducibility of a quasi-separated rigid analytic space and that of an irreducible component of such a space. He also shows that this notion is well-behaved under the functor $\Xscr\mapsto \Xscr^\text{rig}$. The following is a slight reformulation of \cite[Theorem 2.3.1]{ConradIrreducible}:

 \begin{theorem} \label{FormalRigidIrr}
   Let $\Xscr\in \AFS_R^\text{rf}$ and $\Xscr_1,\hdots,\Xscr_m$ be the irreducible components of $\Xscr$. The closed rigid analytic subvarieties $\Xscr_i^\text{rig}$ of $\Xscr^\text{rig}$ are the irreducible components of $\Xscr^\text{rig}$. 
 \end{theorem}
 \begin{proof}
   In fact, \cite[Theorem 2.3.1]{ConradIrreducible} says the following: Let $C\defeq \Gamma(\Xscr,\CO_\Xscr)$ and let $\tilde C$ be the normalization of $C$, that is the integral closure of $C$ in its total ring of fractions. Let $I_1,\hdots,I_n$ be the preimages in $C$ of the minimal prime ideals of $\tilde C$. Let $\Yscr_1,\hdots,\Yscr_n$ be the closed formal subschemes of $\Xscr$ defined by the ideals $I_i$. Then $\Yscr_1^\text{sch},\hdots,\Yscr_n^\text{rig}$ are the irreducible components of $\Xscr^\text{rig}$.

To show that this is the same statement as the one we want to prove, we need to show that the preimages of the minimal prime ideals of $\tilde C$ are exactly the minimal prime ideals of $C$. This is true for any normalization homomorphism $C\to \tilde C$.
 \end{proof}

 \begin{corollary} \label{FormalRigIrr2}
   A formal scheme $\Xscr\in\AFS_R^\text{rf}$ is irreducible if and only if the rigid analytic space $\Xscr^\text{rig}$ is irreducible.
 \end{corollary}
 \begin{proposition} \label{GeomIrrComp}
   Let $\Xscr\in \AFS_R^\text{rf}$ be non-empty. There exists a finite field extension $K'\subset \bar K$ of $K$ with valuation ring $R'$ such that the irreducible components of $\Xscr_{\Spf(R')}$ are geometrically irreducible. 
 \end{proposition}
 \begin{proof}[Proof of Proposition \ref{GeomIrrComp}]
   In \cite[Section 3.4]{ConradIrreducible}, Conrad calls a quasi-separated rigid analytic space $X$ over $K$ geometrically irreducible if for all completely valued overfields $K'$ of $K$ the rigid analytic space $X_{K'}$ is irreducible. By \cite[Theorem 3.4.3]{ConradIrreducible}, for any quasi-separated rigid analytic space $X$ over $K$ having finitely many irreducible components, there exists a finite field extension $K'\subset \bar K$ of $K$ with valuation ring $R'$ such that $\Xscr_{K'}^\text{rig}$ has finitely many irreducible components which are geometrically irreducible. Using the compatibility of the functor $\Xscr\mapsto \Xscr^\text{rig}$ with base change to finite extensions of $K$ (c.f. \cite[7.2.6]{deJong1}) the claim thus follows from Theorem \ref{FormalRigidIrr}.

 \end{proof}
 \subsubsection{Formal Schematic Image}

 \begin{definition}
   Let $f\colon \Xscr\to \Xscr'$ be a morphism of affine formal schemes. We define the \emph{formal schematic image} $f(\Xscr)$ of $f$ to be the intersection of all closed formal subschemes of $\Xscr'$ through which $f$ factors.
 \end{definition}

 Thus $f(\Xscr)$ is the smallest closed formal subscheme of $\Xscr'$ through which $f$ factors. If $\Xscr=\Spf(C)$ and $\Xscr'=\Spf(C')$ then the ideal corresponding to $f(\Xscr)$ is the kernel of the homomorphism $C'\to C$ corresponding to $\Xscr\to\Xscr'$. 

 \begin{lemma} \label{PreimageImage}
   Let $\Gscr, \Gscr'$ be connected $p$-divisible groups over $\Spf(R)$ considered as formal schemes. Let $f\colon \Gscr\to\Gscr'$ be an isogeny and $\Yscr\subset \Gscr'$ be a closed formal subscheme. Let $\Xscr\defeq f^{-1}(\Yscr)\defeq \Yscr\times_{\Gscr'} \Gscr$ be its preimage in $\Gscr$. Then $\Yscr$ is the formal schematic image of $\Xscr$ in $\Gscr'$. 
 \end{lemma}
 \begin{proof}
   Let $C\defeq \Gamma(\Gscr,\CO_\Gscr)$ and $C'\defeq \Gamma(\Gscr',\CO_\Gscr')$. By \cite[Proposition 4.4]{MessingCrystals}, the rings $C$ and $C'$ are isomorphic to $\Spf(R[[x_1,\hdots,x_n]])$ for some $n\geq 0$. First we want to show that $C'\to C$ is flat. Since $R[[x_1,\hdots,x_n]]$ is a regular local ring, by \cite[Theorem 18.16]{Eisenbud} for this it suffices to show that $\dim(C')=\dim(C/\Fn' C)+\dim(C')$ where $\Fn'$ is the maximal ideal of $C'$. This follows from the fact that $C/\Fn' C$ is finite over $k$. Thus as a finite flat module over the local ring $C'$, the ring $C$ is finite free over $C'$.

Let $I'\subset C'$ be the ideal defining $\Yscr$. Then $\Xscr$ is the formal spectrum of $C'/I'\hat\otimes_{C'} C$. Since $C\to C'$ is finite free we have $C'/I'\hat\otimes_{C'} C =C'/I'\otimes_{C'}C\cong C/I'C$. 

Now let $\Yscr'$ be the formal schematic image of $\Xscr$ in $\Gscr'$. It is contained in $\Yscr$. Thus it is defined by an ideal $\tilde I'$ containing $I'$. Its preimage in $\Gscr$ must coincide with $\Xscr$. Thus the induced homomorphism $C'/I'\otimes_{C'}C\to C'/\tilde I'\otimes_{C'}C$ is an isomorphism. Since $C'\to C$ is finite free this implies that $C'/I'\to C'/\tilde I'$ is an isomorphism. This means $\Yscr'=\Yscr$. 
 \end{proof}

 \subsection{Base Change and Formal Schematic Closure}
  \begin{proposition} \label{BaseChangeClosure}
    Let $T\subset \Fm^{\oplus n}$. Let $R'$ be the valuation ring of a complete overfield $K'\subset \hat{\bar K}$ of $K$. Let $I\subset C_{n,0}$ and $I'\subset C'_{n,0}$ be the ideals consisting of those power series which vanish at all elements of $T$. Then $IC'_{n,0}=I'$.
 \end{proposition}
 \begin{proof}
  The ideal $I$ is characterized by the left exact sequence
\begin{equation*}
  0 \to I \to C_{n,0} \ltoover{(\operatorname{ev}_t)_{t\in T}} \prod_{t\in T}R
\end{equation*}
where for $t\in T$ we let $\operatorname{ev}_t\colon C_{n,m} \to R$ be the function given by evaluation at $t$. 

Let $D\cong C_{n,0}/I$ be the image of $C_{n,0}$ in $\prod_{t\in T}R$ endowed with the quotient topology from $C_{n,0}$. This is a topological $R$-algebra of formally finite type by Proposition \ref{CProps}. Lemma \ref{ClosedSubschemeBaseChange} yields an exact sequence 
\begin{equation} \label{BlaSES}
  0 \to IC'_{n,0} \to C'_{n,0} \to D\hat\otimes_R R' \to 0.
\end{equation}

 Let $D_i$ be the kernel of $D\to \prod_{t\in T}R\to \prod_{t\in T}R/\Fm^i$. We claim that there exists a sequence $(j(i))_{i\geq 0}$ of positive integers going to infinity such that $D_i\subset J_{n,0}^{j(i)}$. If not there exists a sequence $(f_i)_{i\geq 0}$ in $D$ which do not converge to zero such that each $f_i$ lies in $D$. The ring $C_{n,0}$, being the inverse limit of the finite rings $C_{n,0}/J^i_{n,0}$, is compact. Hence so is the quotient $D$ of $C_{n,0}$. Hence after passing to a subsequence we may assume that the $f_i$ converge to an element $f\in D$. It follows that the images of the $f_i$ in $\prod_{t\in T}R$ converge to the image of $f$ with respect to the topology on $\prod_{t\in T}R$ defined by the ideals $\prod_{t\in T}\Fm^i$. But by assumption the images of the $f_i$ converge to zero in $\prod_{t\in T}R$. Hence $f$ is zero, which is a contradiction. Thus there exists a sequence $(j(i))_{i\geq 0}$ as above. Since $J_{n,0}^iD\subset D_i$ this shows that the topology on $D$ defined by the $D_i$ is the $J_{n,0}$-adic topology. Thus $D\hat\otimes_R R'$ can be written as $\varprojlim_i D/D_i\otimes_R R'$. 

The inclusion $D\into \prod_{t\in T}R$ induces injections $D/D_i\into (\prod_{t\in T}R/\Fm^i)$ and hence injections $D/D_i\otimes_R R' \into (\prod_{t\in T}R/\Fm^i)\otimes_RR'$. Hence we get a homomorphism
 \begin{equation} \label{BlaComp}
    D\hat\otimes_R R'=\varprojlim_i(D/D_i\otimes_R R')\to \varprojlim_i( (\prod_{t\in T}R/\Fm^i)\otimes_R R')\to \varprojlim_i(\prod_{t\in T}R'/(\Fm R')^i)\cong \prod_{t\in T} R'
 \end{equation}
 and by a direct verification the composition of this homomorphism with the homomorphism $C'_{n,0} \to D\hat\otimes_R R'$ from \eqref{BlaSES} is the homomorphism which evaluates a power series at elements of $T$. 

We want to show that the homomorphism \eqref{BlaComp} is injective. That the first arrow in \eqref{BlaComp} is injective follows from the choice of the $D_i$, the flatness of $R\to R'$ and the left exactness of the inverse limit functor. By the left exactness of inverse limits, in order to prove injectivity of the second arrow in \eqref{BlaComp}, it is enough to show that $(\prod_{t\in T}R/\Fm^i)\otimes_R R'\to \prod_{t\in T}R'/(\Fm R')^i$ is injective for all $i$. Since any element of $(\prod_{t\in T}R/\Fm^i)\otimes_R R'$ is contained in the image of $(\prod_{t\in T}R/\Fm^i)\otimes_R R''$ for the valuation ring $R''$ of a finite field extension $K''\subset K'$ of $K'$ for this step we may assume that $K'$ is finite over $K$. But then $R'$ is finite free over $R$ and hence $(\prod_{t\in T}R/\Fm^i)\otimes_R R'\to \prod_{t\in T}R'/(\Fm R')^i$ is even an isomorphism in this situation. 

Thus by combining \eqref{BlaSES} and \eqref{BlaComp} we get a left exact sequence
\begin{equation*}
  0\to IC'_{n,0} \to C'_{n,0} \ltoover{(\operatorname{ev}_t)_{t\in T}} \prod_{t\in T}R'
\end{equation*}
which shows $IC'_{n,0}=I'$. 

 \end{proof}

\begin{definition}
  Let $R'$ be the valuation ring of a complete overfield $K'\subset \hat{\bar K}$ of $K$. Let $\Xscr$ be an affine formal scheme over $\Spf(R')$ and $T\subset \Xscr(R')$. We define the \emph{formal schematic closure} of $T$ in $\Xscr$ to be the intersection of all closed subschemes $\Yscr$ of $\Xscr$ for which $T\subset \Yscr(R')$.

If $\Xscr'$ is the formal schematic closure of $T$, then we say that $T$ is \emph{formal-schematically dense} in $\Xscr'$.
\end{definition}
Thus the formal schematic closure of $T$ is the smallest closed subscheme of $\Xscr$ which contains $T$.
\begin{corollary} \label{BaseChangeSchematicClosure}
  Let $\Xscr\in \AFS_R$ and $T\subset \Xscr(R)$. Let $\Yscr$ be the formal schematic closure of $T$ in $\Xscr$. Let $R'$ be the valuation ring of a complete overfield $K'\subset \hat{\bar K}$ of $K$. Then $\Yscr_{\Spf(R')}$ is the formal schematic closure of $T\subset \Xscr(R)\subset \Xscr(R')$ inside $\Xscr_{\Spf(R')}$. 
\end{corollary}
\begin{proof}
  It suffices to prove this for $\Xscr=\Spf(C_{n,m})$ in which case it is a reformulation of Proposition \ref{BaseChangeClosure}.
\end{proof}

 \subsection{Transporters}

Let $\Gscr$ be a group object in the category $\AFS_R$. We will need the existence of (strict) transporters in $\Gscr$: 
\begin{construction} \label{TransConstruction}
Let $\Xscr, \Yscr\subset \Gscr$ be two closed formal subschemes. For $i\geq 0$ let $\Gscr_i$ be the $i$-th infinitesimal neighborhood of the zero section in $\Gscr$ and let $\Xscr_i, \Yscr_i \defeq \Gscr_i\cap \Xscr, \Gscr_i\cap \Yscr$. These are finite schemes over $R/\Fm^i$. The group structure on $\Gscr$ makes $\Gscr_i$ into a group scheme over $R/\Fm^i$ and $\Xscr_i$ and $\Yscr_i$ are closed subschemes of $\Gscr_i$. For $i\geq 0$ let $\Trans_{\Gscr_i}(\Xscr_i,\Yscr_i)$ be the strict transporter in $\Gscr_i$, that is the closed subscheme of $\Gscr_i$ whose points are those points $g$ of $\Gscr_i$ for which $\Xscr_i+g=\Yscr_i$. It exists by \cite[Exemple VI.6.4.2 e)]{SGA3I}. Then for each $i$ one has a decreasing sequence $(\Trans_{\Gscr_{i+j}}(\Xscr_{i+j},\Yscr_{i+j})\cap \Gscr_i)_{j\geq 0}$ of subschemes of $\Gscr_i$. By noetherianity this sequence stabilizes; let $\Trans_\Gscr(\Xscr,\Yscr)_i$ be its eventual value. Then $\Trans_\Gscr(\Xscr,\Yscr)_{i+1}\cap \Gscr_i=\Trans_\Gscr(\Xscr,\Yscr)_i$ for all $i$. Hence the inductive limit of these schemes is a closed formal subscheme $\Trans_\Gscr(\Xscr,\Yscr)$ of $\Gscr$.
\begin{proposition} \label{TransDesc}
  Let $\Xscr,\Yscr$ be closed formal subschemes of $\Gscr$ and let $R'$ be the valuation ring of a finite field extension $K'$ of $K$. Then $\Trans_\Gscr(\Xscr,\Yscr)(R')=\{g\in \Gscr(R')\mid g+\Xscr_{\Spf(R')}=\Yscr_{\Spf(R')}\}$.
\end{proposition}
\begin{proof}
  Let $g\in \Gscr(R')$ and let $\Fm'$ be the maximal ideal of $R'$. For $j\geq 0$ we denote the ring $R'/(\Fm')^j$ by $R'_j$ and let $g_j$ be the image of $g$ in $\Gscr(R'_j)$. Then
\begin{align*}
  g \in \Trans_\Gscr(\Xscr,\Yscr)(R') &\iff \forall j\geq 0\colon g_j\in \Trans_\Gscr(\Xscr,\Yscr)_j(R'_j) \\
  &\iff \forall j\geq 0 \; \forall i \gg j\colon g_j + (\Xscr_i)_{R'_j}=(\Yscr_i)_{R'_j} \\
  & \iff g+\Xscr_{\Spf(R')}=\Yscr_{\Spf(R')}.
\end{align*}

\end{proof}
\end{construction}
\subsection{Descent}
Let $R'$ be the valuation ring of a complete overfield $K'\subset \hat{\bar K}$ of $K$. Let $\Xscr\in\AFS_{R}$ and $\Xscr'\subset \Xscr_{\Spf(R')}$ a closed formal subscheme. If there exists a closed formal subscheme $\Xscr''$ of $\Xscr$ such that $\Xscr''_{\Spf(R')}=\Xscr'$ then it follows from Lemmas \ref{BaseChangeExact}, \ref{BaseChangeFlat} and \ref{ClosedSubschemeBaseChange} that such an $\Xscr$ is unique. In this case we will say that \emph{$\Xscr'$ is defined over $R$}. 

In the following we will always endow $k[[x_1,\hdots,x_n]]$ with the $(x_1,\hdots,x_n)$-adic topology. Let $\Xscr=\Spf(C)$ be a formal scheme over $\Spec(k)$ which is isomorphic to $\Spf(k[[x_1,\hdots,x_n]])$. We are interested in closed formal subschemes of $\Xscr$ and their base change to $\Spf(R)$, which lie in $\AFS_R$. 

\begin{lemma} \label{ClosedSubschemeBaseChange2}
  Let $\Xscr'$ be a closed formal subscheme of $\Xscr$ defined by an ideal $I$ of $C$. Then $\Xscr'_{\Spf(R)}$ is the closed formal subscheme of $\Xscr_{\Spf(R)}$ defined by the ideal $I(C\hat\otimes_k R)$ of $C\hat\otimes_k R$. This ideal is equal to $I\hat\otimes_k R$.
\end{lemma}
\begin{proof}
  Since the homomorphisms $k[x_1,\hdots,x_n]/(x_1,\hdots,x_n)^i\to R[x_1,\hdots,x_n]/(x_1,\hdots,x_n)^i$ are faithfully flat, this follows from Lemma \ref{BaseChangeExact} in the same way as Lemma \ref{ClosedSubschemeBaseChange}.
\end{proof}
Let $\Xscr'\subset \Xscr_{\Spf(R)}$ a closed formal subscheme. If there exists a closed formal subscheme $\Xscr''$ of $\Xscr$ such that $\Xscr''_{\Spf(R)}=\Xscr'$ then it follows from Lemmas \ref{BaseChangeExact} and \ref{ClosedSubschemeBaseChange2} that such an $\Xscr$ is unique. In this case we will say that \emph{$\Xscr'$ is defined over $k$}. 

Let $F\colon \Xscr\to \Xscr$ be the Frobenius endomorphism of $\Xscr$ with respect to $k$. If $\Xscr'$ is a closed formal subscheme of $\Xscr_{\Spf(R)}$ defined by formal power series $\{f_i\mid i\in I\}\subset R[[x_1,\hdots,x_n]]$, then the formal schematic image $F(\Xscr)$ of $\Xscr$ under $F$ is defined by the power series obtained by applying $F$ to the coefficients of the $f_i$.
\begin{lemma} \label{FormalFieldDefinition}
For $i\geq 0$ let $K^{p^i}$ be the field consisting of $p^i$-th powers of elements of $K$ and $R^{p^i}$ the valuation ring of $K^{p^i}$. Let $\Xscr'$ be a closed formal subscheme of $\Xscr_{\Spf(R)}$. If $\Xscr'$ is defined over $R^{p^i}$ for all $i\geq 0$, then $\Xscr'$ is defined over $k$.
\end{lemma}
\begin{proof}
We may assume that $\Xscr=\Spf(k[[x_1,\hdots,x_n]])$. Let $i\geq 0$. As any element of $R^{p^i}$ is congruent to an element of $k$ modulo $\Fm^{p^i}$, the fact that $\Xscr'$ can be defined over $R^{p^i}$ implies that the intersection of $\Xscr'$ with $\Spec((R/\Fm^{p^i})[x_1,\hdots,x_n]/(x_1,\hdots,x_n)^{p^i})\subset \Xscr$ can be defined over $k$. As $\Xscr'$ is the direct limit of these intersections, the claim follows by varying $i$.
\end{proof}

\begin{lemma} \label{FrobeniusDescent}
  A closed formal subscheme $\Xscr'$ of $\Xscr_{\Spf(R)}$ is defined over $k$ if and only if $F(\Xscr)=\Xscr$.
\end{lemma}
\begin{proof}
  The ``only if'' direction is straightforward. For the other direction, the fact that $F^i(\Xscr)=\Xscr$ implies that $\Xscr$ is defined over $R^{p^i}$ for all $i\geq 0$. Thus we can conclude using Lemma \ref{FormalFieldDefinition}.

\end{proof}

\begin{proposition} \label{FormalDescent}
   Let $\Gscr$ be a formal group scheme over $k$ which as a formal scheme is isomorphism to $\Spf(k[[x_1,\hdots,x_n]])$. Let $\Xscr\subset \Gscr_{\Spf(R)}$ be a closed formal subscheme. If for each $i\geq 0$ there exists a finite field extension $K'\subset \bar K$ of $K$ with valuation ring $R'$ and $g\in \Gscr(R')$ such that $\Xscr_{\Spf(R')}+g$ is defined over $R^{p^i}$, then there exists a finite field extension $K'\subset \bar K$ of $K$ with valuation ring $R'$ and $g\in \Gscr(R')$ such that $\Xscr_{\bar R'}+g$ is defined over $k$.
\end{proposition}
\begin{proof}
 We consider the closed formal subscheme $\Trans_{\Gscr_{\Spf(R)}}(\Xscr,F(\Xscr))\subset \Gscr_{\Spf(R)}$ given by Construction \ref{TransConstruction}. Let $i \geq 0$ and pick $R'$ as in the claim together with $g\in \Gscr(R')$ such that $\Xscr_{\Spf(R')}+g$ is defined over $R^{p^i}$. Let $\pi\in R$ be a uniformizer. By identifying $R$ with $k[[\pi]]$ and considering defining equations for $\Xscr_{\Spf{R'}}+g$ with coefficients from $k[[\pi^{p^i}]]$ one sees that $\Xscr_{R'}+g\equiv F(\Xscr_{R'}+g) \pmod{\pi^{p^i}}$. Thus $g-F(g)\in \Trans_\Gscr(\Xscr,\Yscr)(\bar R/(\Fm\bar R)^i$). Thus Proposition \ref{FormalSchemeNonempty} implies that there exists $g' \in \Trans_\Gscr(\Xscr,\Yscr)(\bar R)$. 

The morphism $\Gscr\to \Gscr, g\mapsto g-F(g)$ is the identity on the tangent space at zero and thus an isomorphism by \cite[A.4.5]{HazewinkelFormalGroups}. Hence $g'$ can be writen as $g''-F(g'')$ for some $g''\in \Gscr(\bar R)$. By Lemma \ref{XRBarUnion} there exists $R'$ as in the claim such that $g''\in\Trans_{\Gscr_{\Spf(R)}}(\Xscr,\Yscr)(R')$. The fact that $g' \in \Trans_\Gscr(\Xscr,\Yscr)(R')$ translates to $\Xscr_{\Spf(R')}+g''=F(\Xscr_{\Spf(R')}+g'')$. By Lemma \ref{FrobeniusDescent} this shows that $\Xscr+g''$ is defined over $k$.
\end{proof}

\subsection{Formal schemes arising from schemes}
\label{FormalSchemeFromScheme}
 Let $\CX$ be a scheme locally of finite type over $R$ together with a $k$-valued point $s\colon \Spec(k)\to \CX$ of the special fiber of $X$. We let $\hat\CX$ the the completion of $\CX$ along the closed subscheme $s$. We denote by $\CO_{\CX,s}$ the stalk of $\CO_\CX$ at the closed point in the image of $s$. The formal scheme $\hat\CX$ is the formal spectrum of the completion $\hat\CO_{\CX,s}$ of the local ring $\CO_{\CX,s}$ with respect to its maximal ideal. 

 \begin{proposition} \label{HatXProps}
   \begin{enumerate}[(i)]
   \item    If $\CX$ is smooth over $R$ at $s$ of relative dimension $n$ the formal scheme $\hat\CX$ is isomorphic to $\Spf(R[[x_1,\hdots,x_n]])$.
   \item The formal scheme $\hat\CX$ is in $\AFS_R$.
   \item The set $\hat\CX(\bar R)$ can be naturally identified with the set of elements of $\CX(\bar R)$ which map the closed point of $\Spec(\bar R)$ to the closed point in the image of $s$. 
   \end{enumerate}
 \end{proposition}
 \begin{proof}
   $(i)$ The fact that $\CX$ is smooth at $s$ as implies (in fact is equivalent to) that there exist an affine neighbourhood $U$ of $s$ and an \'etale morphism $U\to \Spec(R[x_1,\hdots,x_n])$ which maps the zero section of the special fiber of $\Spec(R[x_1,\hdots,x_n])$ to $s$ (c.f. \cite[Tag 054L ]{stacks-project}). Such an \'etale morphism induces a finite \'etale morphism $\hat\CO_{\CX,s}\to R[[x_1,\hdots,x_n]]$. Since $\hat\CO_{\CX,s}$, being a complete local ring, is Henselian, the fact that both $\hat\CO_{\CX,s}$ and $R[[x_1,\hdots,x_n]]$ have residue field $k$ implies this is an isomorphism.

$(ii)$ Using a closed embedding of an affine neighborhoud of $s$ into $\Spec(R[x_1,\hdots,x_n])$ for some $n \geq 0$ one gets a closed embedding of $\hat\CX$ into $\Spf(R[[x_1,\hdots,x_n]])$. 

$(iii)$ This follows from the fact that $\hat\CX=\Spf(\hat\CO_{\CX,s})$.
 \end{proof}

Let $\CX'$ be second scheme locally of finite type over $R$ together with a $k$-valued point $s'\colon \Spec(k)\to\CX'$. For a morphism $h\colon \CX\to \CX'$ over $R$ which is maps $s$ to $s'$ we denote by $\hat h$ the induced momorphism $\hat\CX\to\hat\CX'$ of formal schemes over $\Spf(R)$.

\begin{proposition} \label{FormalPointsDense}
  Let $\CX$ be a reduced scheme which is flat and of finite type over $R$ and let $s\colon \Spec(k)\to\CX$ a $k$-valued point of $\CX$.
  \begin{enumerate}[(i)]
  \item The formal scheme $\hat\CX$ is reduced and flat over $R$.
  \item Assume that $\CX$ is integral and let $\Xscr$ be an irreducible component of $\hat\CX$. The set $\Xscr(\bar R)\subset \CX(\bar R)$ is schematically dense in $\CX_{\bar R}$.
    \end{enumerate}
\end{proposition}
\begin{proof}
  $(i)$ Since $\CX$ is reduced, so is the local ring $\CO_{\CX,s_0}$. The ring $\CO_{\CX,s_0}$ is also excellent. Since the completion of any excellent reduced local ring is reduced (c.f. \cite[7.8.3]{EGA4II}) the formal scheme $\hat\CX$ is reduced. Flatness follows from the flatness of $R\to \CO_{\CX,s_0}$ and the flatness of $\CO_{\CX,s}\to\hat\CO_{\CX,s}$.

$(ii)$ Let $\CY\subset \CX$ be the schematic closure of $\Xscr(\bar R)\subset \CX(\bar R)$ and let $\CI\subset \CO_\CX$ be the sheaf of ideals defining $\CY$. 

 As $\CO_{\CX,s}$ is a Noetherian local ring, the homomorphism  $\CO_{\CX,s}\to\hat\CO_{\CX,s}$ is faithfully flat and for any finitely generated $\CO_{\CX,s}$-module $M$, its completion with respect to the topology induced by the maximal ideal of $\CO_{\CX,s}$ is isomorphic to $\hat\CO_{\CX,s}\otimes_{\CO_{\CX,s}} M$. By applying this to $M=\CO_{\CY,s}$ one sees that $\hat\CY$ is the formal closed subscheme of $\hat\CX$ corresponding to the ideal $\CI_s\hat\CO_{\CX,s}$. Since by construction $\Xscr(\bar R)\subset \hat\CY(\bar R)$ Proposition \ref{FormalJacobson} implies $\Xscr\subset \hat\CY$. Thus $\CI_s\hat\CO_{\CX,s}$ is contained in a minimal prime ideal of $\hat\CO_{\CX,s}$ and hence the rings $\hat \CO_{\CX,s}$ and $\hat\CO_{\CY,s}$ have the same dimension. 

Using the flatness of $\CO_{\CX,s}\to \hat\CO_{\CX,s}$ and the fact that the maximal ideal of $\hat\CO_{\CX,s}$ is generated by the image of the maximal ideal of $\CO_{CX,s}$, Theorem 10.10 of \cite{Eisenbud} implies that $\dim(\CO_{\CX,s})=\dim(\hat\CO_{\CX,s})$. Analogously we get $\dim(\CO_{\CY,s})=\dim(\hat\CO_{\CY,s})$. Thus $\dim(\CY)\geq \dim(\CO_{\CY,s})=\dim(\hat\CO_{\CY,s})=\dim(\hat\CO_{\CX,s})=\dim(\CO_{\CX,s})$. Since $\CX$ is irreducible $\dim(\CO_{\CX,s})=\dim(\CX)$ and thus we get $\dim(\CY)=\dim(\CX)$ which using the irreducibility of $\CX$ implies $\CY=\CX$.
\end{proof}

\begin{proposition} \label{CompletionImage}
  Let $\CX, \CX$ be schemes locally of finite type over $R$. Let $s,s'$ be $k$-valued points of $\CX,\CX'$ and let $f\colon \CX\to\CX'$ be morphism over $R$ which maps $s$ to $s'$ and is flat at $s$. Let $\Xscr'$ be the formal schematic image of the induced morphism $\hat f\colon \hat\CX\to \hat\CX'$. Every irreducible component of $\Xscr'$ is an irreducible component of $\hat\CX'$. 
\end{proposition}
\begin{proof}
  By the assumption of $f$ the induced homomorphism $\CO_{\CX',s'}\to \CO_{\CX,s}$ is flat. First we want to prove that the induced homomorphism $\hat\CO_{\CX,s'}\to\hat\CO_{\CX,s}$ is also flat. Let $\Fn'$ be the maximal ideal of $\hat\CO_{\CX',s'}$ and $\hat\Fn'$ its completion, which is the maximal ideal of $\hat\CO_{\CX',s'}$ . By the local criterion for flatness (see \cite[Theorem 6.8]{Eisenbud}) is suffices to show that $\Tor_1^{\hat\CO_{\CX',s'}}(\hat\CO_{\CX',s'}/\hat\Fn',\hat\CO_{\CX,s})=0$. Note that since $\hat\Fn'=\Fm\hat\CO_{\CX',s'}$ we have $\hat\CO_{\CX',s'}/\hat\Fn'\cong \CO_{\CX',s'}/\Fn'\otimes_{\CO_{\CX',s'}} \hat\CO_{\CX',s'}$. Using this and the flatness of $\CO_{\CX',s'}\to\hat\CO_{\CX',s'}$ the Proposition 3.2.9 of \cite{Weibel} on flat base change for Tor says
  \begin{equation*}
    \Tor_1^{\hat\CO_{\CX',s'}}(\hat\CO_{\CX',s'}/\hat\Fn',\hat\CO_{\CX,s})=\Tor_1^{\CO_{\CX',s'}}(\CO_{\CX',s'}/\Fn',\hat \CO_{\CX,s}).
  \end{equation*}
The second term of this equation is zero since the homomorphism $\CO_{\CX',s'}\to \hat\CO_{\CX,s}$ is flat, being the composition of the two flat homomorphisms $\CO_{\CX',s'}\to \CO_{\CX,s}$ and $\CO_{\CX,s}\to\hat\CO_{\CX,s}$. Thus $\hat\CO_{\CX,s'}\to\hat\CO_{\CX,s}$ is also flat.

That the homomorphism $\hat\CO_{\CX',s'}\to\hat\CO_{\CX,s}$ is flat implies by \cite[Lemma 10.11]{Eisenbud} that it has the going down property, that is for any prime ideal $\Fp$ of $\hat\CO_{\CX,s}$ and any prime ideal $\Fq\subset \hat\CO_{\CX',s'}\cap \Fp$ there exists a prime ideal $\Fp'\subset \Fp$ of $\hat\CO_{\CX,s}$ such that $\Fq=\hat\CO_{CX',s'}\cap \Fp'$. By applying this to a minimal prime ideal $\Fp$ of $\hat\CO_{\CX,s}$ one see that its pullback $\hat\CO_{\CX',s'}\cap \Fp$ is a minimal prime ideal in $\hat\CO_{\CX',s'}$. 

This means that the formal schematic image of any irreducible component of $\hat\CX$ is an irreducible component of $\hat\CX'$. This implies the claim.
\end{proof}

 We will in particular apply the above to a smooth group scheme $\CA$ over $R$ with $s$ the zero section of the special fiber. Then $\CA$ is a formal group scheme, which, as a formal scheme, is isomorphic to $\Spf(R[[x_1,\hdots,x_n]])$. Note that often one denotes by $\hat\CA$ the completion of $\CA$ along its zero section. This is a formal scheme over $\Spec(R)$, and $\hat\CA$ as we define it is the base change of this formal scheme to $\Spf(R)$ along the morphism $\Spf(R)\to\Spec(R)$ given by the identity homomorphism $R\to R$. 

If $[p]\colon \hat\CA\to\hat\CA$ is an epimorphism, then $\hat\CA$ is a $p$-divisible group over $\Spf(R)$ (c.f. \cite[Corollary 4.5]{MessingCrystals}). More precisely, it follow from [loc. cit.] that in this case $\hat\CA$ is equal to the pullback to $\Spf(R)$ of the connected part $\CA[p^\infty]^\circ$ of the $p$-divisible group $\CA[p^\infty]$ of $\CA$. 

For a closed subscheme $\CX$ of $\CA$ containing the zero section, we denote by $\hat\CX$ its completion along the zero section of the special fiber.

\section{Special Subvarieties}
\label{sec:SpecialSubvarieties}
For any group scheme $\CA$ over a scheme $S$ and any closed subscheme $\CX$ of $\CA$ we denote by $\Stab_\CA(\CX)$ the functor which associates to any scheme $S'$ over $S$ the set $\{a\in \CA(S')\mid a+\CX_{S'}=\CX_{S'}\}$ and acts on morphisms by pullbacks. In case $S$ is a field or a valuation ring, which will be the only relevant cases for us, this functor is representable by a closed subscheme of $\CA$, c.f. \cite[Exp. VIII, Ex. 6.5(e)]{SGA3II}.

For an extension $L\into L'$ of algebraically closed fields, recall the the notion of the $L'/L$-trace of an abelian variety $A$ over $L'$ (c.f. \cite{ConradTrace}): This is an abelian variety over $L$, which we will denote $\Tr_{L'/L}A$, together with a homomorphism $\tau\colon (\Tr_{L'/L}A)_{L'}\to A$ which satisfies the following universal property: For each abelian variety $B$ over $L$ together with a homomorphism $f \colon B_{L'}\to A$, there exists a unique homomorphism $g\colon B\to \Tr_{L'/L}A$ (defined over $L$) such that $f=\tau\circ g_{L'}$. For $L\subset L'$ algebraically closed this trace always exists (see \cite[6.2]{ConradTrace}) and the map $\tau$ has finite kernel (see \cite[6.4]{ConradTrace}). Thus roughly speaking $\Tr_{L'/L}A$ is the largest subobject of $A$ which can be defined over $L$. It is  determined up to unique isomorphism and functorial in $A$.

Let $L$ be an algebraically closed field of characteristic $p> 0$ and $A$ a semiabelian variety over $L$.

\begin{definition} \label{SpecialDef}
  We call a subvariety $X$ of $A$ \emph{special} if $X$ is irreducible and there exist a semiabelian variety $B$ over $\bar \BF_p$, a subvariety $Y$ of $B$ over $\bar \BF_p$, a homomorphism $h\colon B_L\to A/\Stab_A(X)$ with finite kernel and an element $a\in (A/\Stab_A(X))(L)$ such that $X/\Stab_A(X)=h(Y)+a$.
\end{definition}
\begin{remark}
  This notion of a special subvariety is equivalent to Hrushovski's \cite{HrushovskiML} notion of a special subvariety, as can be shown using Lemma \ref{SpecialImage} below.
\end{remark}
\begin{lemma} \label{SpecialImage}
  Let $B$ be a semiabelian variety over $L$, let $X\subset A$ an irreducible subvariety and let $f\colon A\to B$ a homomorphism. If $X$ is a special subvariety of $A$, then $f(X)$ is a special subvariety of $B$. If $f$ has finite kernel and $f(X)$ is a special subvariety of $B$, then $X$ is a special subvariety of $A$.
\end{lemma}
\begin{proof}
The homomorphism $f$ induces a homomorphism $\bar f\colon A/\Stab_A(X)\to B/\Stab_B(f(X))$. 

  If $X$ is special, then there exist a semiabelian variety $C$ over $\bar \BF_p$, a subvariety $Y$ of $C$ over $\bar \BF_p$, a homomorphism $h\colon C_L\to A/\Stab_A(X)$ with finite kernel and an element $a\in (A/\Stab_A(X))(L)$ such that $X/\Stab_A(X)=h(Y)+a$. Let $D$ be the connected component of the identity of the kernel of $\bar f\circ h$ equipped with the reduced scheme structure. This is a semiabelian subvariety of $B_L$. Thus it is defined over $\bar\BF_p$. The homomorphism $\bar f\circ h$ induces a homomorphism $h'\colon C/D\to B/\Stab_B(f(X))$ with finite kernel such that $f(X)/\Stab_B(f(X))=h'(Y/D)+\bar f(a)$. Thus $f(X)$ is special in $B$.

Now let $f$ have finite kernel and $f(X)$ be special in $B$. After replacing $B$ by the image of $f$, we may assume that $f$ is an isogeny. Then $\bar f$ is also an isogeny. Pick an isogeny $g\colon B/\Stab_B(f(X))\to A/\Stab_A(X)$ such that $g\circ \bar f=[n]$ for some $n\in\BZ^{\not= 0}$. There exist a semiabelian variety $C$ over $\bar \BF_p$, a subvariety $Y$ of $C$ over $\bar \BF_p$, a homomorphism $h\colon C_L\to B/\Stab_X(f(X))$ with finite kernel and an element $b\in (B/\Stab_B(f(X)))(L)$ such that $f(X)/\Stab_B(f(X))=h(Y)+b$. Then $[n](X/\Stab_A(X))=g(h(Y))+g(b)$. Let $Y'$ be an irreducible component of $[n]^{-1}(Y)$. Then $g(h(Y'))+g(b)$ is an irreducible component of $[n]^{-1}([n](X/\Stab_A(X)))$. Since the $n$-torsion points of $(A/\Stab_A(X))(L)$ act transitively on the irreducible components of $[n]^{-1}([n](X/\Stab_A(X)))$, it follows that $X/\Stab_A(X)$ is a translate of $g(h(Y'))$. This shows that $X$ is special in $A$. 
\end{proof}
\begin{lemma} \label{SpecialBaseChange}
  Let $L'$ be an algebraically closed overfield of $L$ and $X\subset L$ a subvariety. Then $X$ is special in $A$ if and only if $X_{L'}$ is special in $A_{L'}$. 
\end{lemma}
\begin{proof}
  The ``only if'' direction follows directly from the definition of a special subvariety. For the other direction see the proof of \cite[Lemma 1.2]{Roessler1}.
\end{proof}
\begin{lemma} \label{ReplaceIsogenous}
  Let $L_0$ be an arbitrary field, let $A_0$ be a semiabelian variety over $L_0$ and let $X_0\subset A_0$ be an irreducible subvariety. Let $\Gamma\subset A_0(L_0)$ be a finitely generated subgroup such that $X_0(L_0)\cap \Gamma$ is Zariski dense in $X_{0}$. For any semiabelian variety $A_0'$ over $L_0$ which is isogenous to $A_0$, there exist a finitely generated subgroup $\Gamma'$ of $A_0'(L_0)$ and an irreducible subvariety $X'_0$ of $A_0'$ such that $X_0'(L_0)\cap \Gamma'$ is Zariski dense in $X_{0}'$ and such that $X_{0}$ is special in $A_{0_0}$ if and only if $X'_{0}$ is special in $A_{0}$.
\end{lemma}
\begin{proof}
  Let $f\colon A_0\to A_0'$ be an isogeny. Take $\Gamma'\defeq f(\Gamma)$ and $X_0'\defeq f(X_0)$. Lemma \ref{SpecialImage} implies that $\Gamma'$ and $X_0'$ have the required properties.
\end{proof}

In the following lemma, which gives an equivalent description of special subvarieties, we denote by $\tau$ the canonical morphism $\Tr_{L/\bar\BF_p}A\to A$.
\begin{lemma} \label{SpecialEquivalence}
  An irreducible subvariety $X\subset A$ is special in $A$ if and only if there exists a closed subvariety $Y\subset \Tr_{L/\bar\BF_p}A$ defined over $\bar\BF_p$ and $a\in A(L)$ such that $X+a=\tau(Y)+\Stab_X(A)$.
\end{lemma}
\begin{proof}
  Because the formation of $\Tr_{L/\bar\BF_p}A$ is functorial in $A$, there is a commutative diagram
  \begin{equation*}
    \xymatrix{
      \Tr_{L/\bar\BF_p}A \ar[r]^\tau \ar[d] & A \ar[d] \\
      \Tr_{L/\bar\BF_p}(A/\Stab_A(X)) \ar[r]^{\tau'} & A/\Stab_A(X) \\
      }
  \end{equation*}
with both semiabelian varieties on the left as well as the morphism between them defined over $\bar\BF_p$ and with $\tau$ and $\tau'$ having finite kernel. The ``if'' direction follows directly from this. For the ``only if'' direction, we assume that $X$ is special. Then there exist $Y\subset B$, $h$ and $a$ as in Definition \ref{SpecialDef}. Since the homomorphism $h$ factors through $\tau'$, we may replace $Y$ by its image in $\Tr_{L/\bar\BF_p}(A/\Stab_A(X))$ and assume $B=\Tr_{L/\bar\BF_p}(A/\Stab_A(X))$. Let $Y'$ be the inverse image of $Y$ in $\Tr_{L/\bar\BF_p}A$. Note that $Y'$ is defined over $\bar\BF_p$. If $a'\in A(L)$ is a lift of $a$ the identity $\tau'(Y)=X/\Stab_A(X)+a$ implies $\tau(Y')+\Stab_A(X)=X+a'$.
\end{proof}

\subsubsection*{Specialness Criteria}

\begin{theorem}[Pink-R\"ossler, see {\cite[Theorem 3.1]{PinkRoessler}}]\label{PRTheorem}
  Let $\phi\colon A\to A$ be an isogeny. Let $X\subset A$ be an irreducible subvariety such that $\phi(X)=X+a$ for some $a\in A(L)$. Then $\phi(\Stab_A(X))=\Stab_A(X)$ and we denote the isogeny $A/\Stab_A(X)\to A/\Stab_A(X)$ induced by $\phi$ by $\bar \phi$.

There exist finitely many homomorphisms $h_{\alpha}\colon A_{\alpha}\to A/\Stab_A(X)$ for certain $\alpha \in \BQ^{\geq 0}$, where the $A_{\alpha}$ are semiabelian varieties endowed with isogenies $\phi_{\alpha}\colon A_{\alpha}\to A_{\alpha}$ satisfying $\bar \phi\circ h_{\alpha}= h_{\alpha}\circ \phi_\alpha$ and irreducible subvarieties $X_{\alpha}\subset A_{\alpha}$ satisfying $\phi_{\alpha}(X_{\alpha})=X_{\alpha}+a_{\alpha}$ for some $a_{\alpha}\in A_{\alpha}(L)$ such that:
  \begin{itemize}
  \item If $\alpha =0$, then $\phi_{\alpha}$ is an automorphisms of finite order of $A_{\alpha}$.
  \item If $\alpha > 0$, then there exist positive integers $r$ and $s$ such that $\alpha=r/s$ and $\phi_{\alpha}^s=\Frob_{p^r}$ for some model of $A_{\alpha}$ over $\BF_{p^r}$.
  \item The morphism
    \begin{equation*}
      h\defeq \sum_\alpha h_{\alpha}\colon \prod_i A_{\alpha} \to A/\Stab_A(X)
    \end{equation*}
    has finite kernel and, for some point $\bar a\in (A/\Stab_A(X))(L)$,
    \begin{equation*}
      X/\Stab_A(X) = \bar a + h(\prod_i X_{\alpha}).
    \end{equation*}
  \end{itemize}
\end{theorem}
We will only need the following consequence of Theorem \ref{PRTheorem}:
\begin{corollary} \label{PRCorollary}
  Let $\phi\colon A\to A$ an isogeny whose minimal polynomial does not have any complex roots which are roots of unity. Let $X\subset A$ be an irreducible variety such that $\phi(X)=X+a$ for some $a\in A(L)$. Then $X$ is a special subvariety of $A$. 
\end{corollary}
\begin{proof}
  The condition on the minimal polynomial of $\phi$ ensures that $\phi$ cannot act as an automorphism of finite order on any subquotient of $A$. Hence the term $A_0$ in Theorem \ref{PRTheorem} does not appear, and it follows using Lemma \ref{SpecialImage} that $X$ is special in $A$. 
\end{proof}
\begin{definition}
  We call a polynomial $f\in \BZ[t]$ \emph{good} if it is monic, if $f(0)\not= 0$ and if no complex root of $f$ is a root of unity.
\end{definition}

\begin{theorem} \label{InfSpecialCrit}
  Let $G\subset A(L)$ be a subgroup and $\Phi\colon G\to G$ an endomorphism such that there exists a good polynomial $f\in\BZ[t]$ which annihilates $\Phi$.

Let $X\subset A$ be an irreducible subvariety. If there exists a subset $T$ of $G\cap X(L)$ which is Zariski dense in $X$ and which satisfies $\Phi(T)\subset T$, then $X$ is special in $A$.
\end{theorem}
\begin{proof}
  Write $f(t)=t^n+\sum_{i=0}^{n-1}a_it^i$ with $a_i\in\BZ$. Let $\phi$ be the endomorphism of $A^n$ defined by the matrix
  \begin{equation*}
    \begin{bmatrix}
      0 & 0 & \cdots & 0 & -a_0 \\
      1 & 0 & \cdots & 0 & -a_1\\
      0 & 1 & \ddots  &  & \vdots \\
      0 & \ddots & 1 & 0 &  \vdots \\
      0 & \cdots & \cdots &1   & -a_{n-1}  \\
    \end{bmatrix},
  \end{equation*}
which satisfies $f(\phi)=0$ and $\phi(a,\Phi(a),\hdots,\Phi^{n-1}(a))=(\Phi(a),\Phi^2(a),\hdots,\Phi^n(a))$ for $a\in G$. Let $X'$ be the Zariski closure of the set $\{(x,\Phi(x),\hdots,\Phi^{n-1}(x))\mid x\in T\}$ in $A^n$. The fact that $\Phi(T)\subset T$ implies that $\phi(X')\subset X'$. Since $a_0=f(0)\not=0$, the above matrix is invertible over $\BQ$ and hence $\phi$ is an isogeny. Hence for each irreducible component $Z$ of $X'$, its image $\phi(Z)$ is also an irreducible component of $X'$. Thus every irreducible component of $X'$ is invariant under some power of $\phi$. Hence by the assumption on $f$ and Corollary \ref{PRCorollary} each irreducible component of $X'$ is special in $A^n$. 

Let $\pi\colon A^n\to A$ be the projection to the first factor. The fact that $T$ is Zariski dense in $X$ implies $\pi(X')=X$. Since $X$ is irreducible, some irreducible component of $X'$ maps onto $X$ under $\pi$. Hence Lemma \ref{SpecialImage} implies that $X$ is special in $A$. 
\end{proof}
\section{The General Setup}
\label{sec:setup}
\subsection{Completely slope divisible $p$-divisible groups} \label{sec:CSD}
First we collect some terminology and facts from \cite{OortZink}.

Let $S$ be a scheme over $\BF_p$. Let $\Frob\colon S\to S$ be the absolute Frobenius morphism $x\mapsto x^p$. For a scheme $G$ over $S$ and $s \geq 1$ we write $G^{(p^s)}=G \times_{S,\Frob^s}S$. We denote by $\F^s\colon G\to G^{(p^s)}$ the Frobenius morphism relative to $S$.

Let $L$ be an algebraically closed field of characteristic $p>0$ and $\Gscr$ a $p$-divisible group over $L$. For a rational number $\lambda\geq 0$, one calls $\Gscr$ \emph{isoclinic of slope} $\lambda$ if there exist integers $r\geq 0$ and $s\geq 0$ such that $\lambda=r/s$ and a $p$-divisible group $\Gscr'$ over $L$ which is isogenous to $\Gscr$ and for which there exists an isomorphism $\psi\colon \Gscr^{(p^s)}\to \Gscr$ making the following diagram commute (c.f. \cite[Section 1]{OortZink}):
  \begin{equation*}
    \xymatrix{ 
      \Gscr \ar^{\F^s}[r] \ar_{[p^{r}]}[rd] & \Gscr^{(p^s)} \ar^{\psi}[d] \\
                                            & \Gscr \\
}
\end{equation*}
Every $p$-divisible group $\Gscr$ over $L$ is isogenous to a direct sum of isoclinic $p$-divisible group and the slopes appearing in such a direct sum determine $\Gscr$ up to isogeny. They are called the slopes of $\Gscr$ and are assembled into the Newton polyon of $\Gscr$ (see e.g. \cite[IV.5]{Demazure}). 
\begin{definition}[{c.f. \cite[Definition 1.2]{OortZink}}] \label{CSDDef}
  \begin{itemize}
  \item   Let $s \geq 1$ and $r_1,\hdots,r_m$ be integers such that $s\geq r_1 > r_2 > \hdots > r_m \geq 0$. A $p$-divisible group $\Gscr$ over a scheme $S$ is said to be \emph{completely slope divisible} with respect to these integers if $\Gscr$ has a filtration by $p$-divisible subgroups
  \begin{equation*}
    0=\Gscr_0 \subset \Gscr_1 \subset \hdots \subset \Gscr_m =\Gscr
  \end{equation*}
such that the following properties hold:
\begin{enumerate}[(i)]
\item  For $i=1,\hdots,m$ the kernel of $[p^{r_i}]\colon \Gscr_i\to \Gscr_i$ is contained in the kernel of $F^s\colon \Gscr_i\to \Gscr_i^{(p^s)}$.
\item For $i=1,\hdots,m$ the kernel of $[p^{r_i}]\colon \Gscr_i/\Gscr_{i-1}\to \Gscr_i/\Gscr_{i-1}$ is equal to the kernel of $F^s\colon \Gscr_i/\Gscr_{i-1}\to (\Gscr_i/\Gscr_{i-1})^{(p^s)}$.
\end{enumerate}

\item A $p$-divisible group $\Gscr$ is completely slope divisible if there exist integers $s\geq r_1 > r_2 > \hdots > r_m \geq 0$ such that $\Gscr$ is completely slope divisible with respect to these integers.
  \end{itemize}
\end{definition}

\begin{remark} 
 Let $\Gscr$ be a $p$-divisible group which is completely slope divisible with respect to $s\geq r_1> r_2 > \hdots >r_m\geq 0$. Note that condition $(ii)$ is equivalent to the existence of isomorphisms $(\Gscr_i/\Gscr_{i+1})^{(p^s)}\isoto \Gscr_i/\Gscr_{i+1}$ such that the following diagram commutes:
  \begin{equation*}
    \xymatrix{ 
      \Gscr_i/\Gscr_{i+1} \ar^{\F^s}[r] \ar_{[p^{r_i}]}[rd] & (\Gscr_i/\Gscr_{i+1})^{(p^s)} \ar[d]^\cong \\
                                            & \Gscr_i/\Gscr_{i+1} \\
}.
\end{equation*}
Thus all geometric fibers of the subquotients $\Gscr_i/\Gscr_{i+1}$ are isoclinic of slope $r_i/s$ and in particular the Newton polygon of $\Gscr$ is constant on $S$.
\end{remark}

\begin{remark} \label{CSDUniqueness}
  Let $\Gscr$ be a $p$-divisible group which is completely slope divisible with respect to $s\geq r_1> r_2 > \hdots >r_m\geq 0$. By the remark after Definition 1.2 of \cite{OortZink} there exists a unique filtration $(\Gscr_i)_{i=0,\hdots m}$ satisfying the conditions above for the given integers $s\geq r_1> r_2 > \hdots >r_m\geq 0$.
\end{remark}

\begin{theorem}[{\cite[Theorem 2.1]{OortZink}}] \label{CSDIsogeny}
  Let $\Gscr$ be a $p$-divisible group over a integral normal Noetherian scheme $S$ with constant Newton polygon. There exists a completely slope divisible $p$-divisible group over $S$ which is isogenous to $\Gscr$. 
\end{theorem}
 We call a scheme $S$ of characteristic $p>0$ \emph{perfect} if for each open set $U$ of $S$ the endomorphism $x\mapsto x^p$ of the ring $\CO_S(U)$ is an isomorphism.
\begin{proposition}[Oort-Zink] \label{SlopeFiltrationSplits}
  Let $\Gscr$ be a $p$-divisible group over a perfect scheme $S$ which is completely slope divisible with respect to integers $s\geq r_1> r_2 > \hdots > r_m\geq 0$. Let $(\Gscr_i)_i$ be a filtration as in Definition \ref{CSDDef} with respect to these integers. The filtration $(\Gscr_i)_i$ splits uniquely, that is there are unique sections $\Gscr_i/\Gscr_{i+1}\to \Gscr_i$ of the quotient maps $\Gscr_i\to \Gscr_i/\Gscr_{i+1}$.
\end{proposition}
\begin{proof}
  This is \cite[Proposition 1.3]{OortZink}. Although the uniqueness of the splittings is not part of the statement there, it is shown in the proof given there.
\end{proof}
\begin{proposition}\label{CSDDescent}
  A $p$-divisible group $\Gscr$ over a scheme $S$ is completely slope divisible if it is so fpqc-locally on $S$.
\end{proposition}
\begin{proof}
  Let $S'\to S$ be an fpqc covering of $S$ such that $\Gscr_{S'}$ is completely slope divisible with respect to integers $s\geq r_1 > r_2 > \hdots > r_m\geq 0$. By Remark \ref{CSDUniqueness}, there are unique subgroups $\Gscr'_i$ of $\Gscr_{S'}$ satisfying the conditions of Definition \ref{CSDDef}. The pullbacks of $(\Gscr'_i)_i$ along the two morphisms $S'\times_S S' \to S'$ both satisfy the conditions of Definition \ref{CSDDef} over $S'\times_S S'$ relative to the above integers. Thus by the uniqueness statement of Remark \ref{CSDUniqueness} these two pullbacks coincide. Hence by fpqc descent the subgroups $\Gscr'_i$ of $\Gscr_{S'}$ arise by base change from subgroups $\Gscr_i$ of $\Gscr$. By fpqc descent the conditions of Definition \ref{CSDDef} hold for $(\Gscr_i)$ if they hold fpqc-locally on $S$. Thus $\Gscr$ is completely slope divisible.
\end{proof}

\begin{proposition} \label{CSDSubgroup}
  Let $\Gscr$ be a $p$-divisible group over a scheme $S$ which is completely slope divisible. Let $\Hscr$ be a $p$-divisible subgroup of $\Gscr$. Then $\Hscr$ and $\Gscr/\Hscr$ are completely slope divisible.
\end{proposition}
\begin{proof}
Let $(\Gscr_i)$ be a filtration as in Definition \ref{CSDDef}. Let $\Hscr_i\defeq \Hscr\cap\Hscr_i$. Then it follows by a direct verification that $(\Hscr_i)$ and $(\Gscr_i/\Hscr_i)$ have the required properties.
\end{proof}

\begin{lemma}[{see \cite[Corollary 1.10]{OortZink}}] \label{FGSConstant}
  Let $G\to S$ be a finite flat group scheme over a connected base scheme $S$. Let $\psi \colon G\isoto G^{(p^s)}$ be an isomorphism. Then there exists a finite \'etale morphism $T\to S$ and a morphism $T\to \Spec(\BF_{p^s})$ such that $G_T$ is obtained by base change from a finite group scheme $H$ over $\Spec(\BF_{p^s})$
  \begin{equation*}
    H\times_{\Spec(\BF_{p^s})} T\isoto G_T
  \end{equation*}
and $\psi$ is induced from the identity on $H$.
\end{lemma}

The argument in the proof of the following proposition is taken from the proof of Proposition 3.1 of \cite{OortZink}.
\begin{proposition} \label{CSDConstant1}
  Let $R$ be a perfect strictly henselian local ring over $\bar\BF_p$ and $\Gscr$ a completely slope divisible $p$-divisible group over $R$. Then there exists a $p$-divisible group $\Gscr_0$ over $\bar\BF_p$ such that $\Gscr_{0,R}$ is isomorphic to $\Gscr$. In case that $\Gscr$ has a single slope, it suffices that $R$ be strictly henselian.
\end{proposition}
\begin{proof}
  In case $\Gscr$ has multiple slopes, by Proposition \ref{SlopeFiltrationSplits} we can write $\Gscr$ as a direct sum of completely slope divisible groups having a single slope. Thus it suffices to treat this case. Then there exist $s\geq r\geq 0$ and an isomorphism $\psi\colon \Gscr^{(p^s)}\to\Gscr$ such that $\psi\circ\F^s=[p^r]$. For $n\geq 0$ denote by $\Gscr(n)$ the kernel of $[p^n]\colon \Gscr\to \Gscr$. Applying Lemma \ref{FGSConstant} to $\Gscr(n)$ and $\psi^{-1}$ we obtain finite group schemes $\Gscr_o(n)$ over $\Spec(\BF_{p^s})$ and isomorphisms 
  \begin{equation*}
    \Gscr(n)\cong \Gscr_0(n)\times_{\Spec(\BF_{p^s})} R.
  \end{equation*}
The inductive limit of the group schemes $\Gscr_0(n)$ is a $p$-divisible group $\Gscr_0$ over $\BF_{p^s}$ which has the required property.
\end{proof}

The following result is probably not new, but we could not find a reference.
\begin{proposition} \label{ConstantHom}
 Let $k$ be an algebraically closed field of positive characteristic. Let $\Gscr$ and $\Hscr$ be $p$-divisible groups over $k$. For any integral scheme $S$ over $k$ the base change map
  \begin{equation*}
    \Hom_k(\Gscr,\Hscr)\to \Hom_S(\Gscr_S,\Hscr_S)
  \end{equation*}
is a isomorphism of $\BZ_p$-modules.
\end{proposition}
\begin{proof}
  Let $k'$ be the function field of $S$. First we claim that for any finite flat group schemes $G$ and $H$ over $S$ the natural map $\Hom_S(G,H)\to \Hom_{k'}(G_{k'},H_{k'})$ is injective: We may assume that $S$ is affine, say $S=\Spec(R)$. Then $G$ and $H$ are the spectrum of finite flat $R$-algebra $A_G$ and $A_H$. Since these are flat over $R$, the homomorphisms $A_G\to A_G\otimes_R k'$ and $A_H\to A_G\otimes_R k'$ are injective. Hence any homomorphism $A_H\to A_G$ of $R$ algebras is determined by its generic fiber $A_H\otimes_R k'\to A_G\otimes k'$. This shows that $\Hom_S(G,H)\to \Hom_{k'}(G_{k'},H_{k'})$ is injective. By applying this to group schemes $\Gscr[p^n]$ and $\Hscr[p^n]$ for $n\geq 0$ one gets that the homomorphism $\Hom_S(\Gscr_S,\Hscr_S)\to \Hom_{k'}(\Gscr_{k'},\Hscr_{k'})$ is injective. Thus we may assume that $S=\Spec(k')$. We may also assume that $k'$ is algebraically closed. 

We use the theory of Dieudonn\'e modules. Denote by $W(k)$ (resp. $W(k')$) the ring of Witt vectors of $k$ (resp. $k'$), by $\sigma$ the lift of Frobenius to these rings and by $B(k)$ (resp. $B(k')$) their quotient field. Let $M(\Gscr)$ and $M(\Hscr)$ be the contravariant Dieudonn\'e modules associated to $\Gscr$ and $\Hscr$. They are free $W(k)$-modules endowed with a $\sigma$-linear self-map $F$ and a $\sigma^{-1}$-linear self-map $V$. 

A homomorphism $\Gscr_{k'}\to \Hscr_{k'}$ corresponds to a $W(k)$-linear homomorphism $M(\Hscr_{k'})=M(\Hscr)\otimes_{W(k)}W(k')\to M(\Gscr_{k'})=M(\Gscr)\otimes_{W(k)}W(k')$ compatible with $V$ and $F$. We need to show that any such homomorphism arises from a homomorphism $M(\Hscr)\to M(\Gscr)$. Since we are dealing with free $W(k)$-modules, it suffices to prove that the induced homomorphism $M(\Hscr)\otimes_{W(k)}B(k')\to M(\Gscr)\otimes_{W(k)}B(k')$ arises from a homomorphism $M(\Hscr)\otimes_{W(k)}B(k)\to M(\Gscr)\otimes_{W(k)}B(k)$. The $B(k)$-vector spaces $M(\Hscr)\otimes_{W(k)}B(k')$ and $M(\Gscr)\otimes_{W(k)}B(k')$ together with the $\sigma$-linear endomorphism induced by $F$ are what is called an $F$-space in \cite[Chapter IV]{Demazure}. By the theorem in \cite[Section IV.4]{Demazure}, each such $F$-space is a direct sum of certain simple $F$-spaces denoted $E_k^\lambda$ for $\lambda\in \BQ^{\geq 0}$. Furthermore, by a proposition in \cite[Section IV.3]{Demazure}, if $\lambda\not=\lambda'$, any homomorphism $E_k^\lambda\to E_k^{\lambda'}$ of $F$-spaces is zero. Hence it suffices to prove that any endomorphism of $E_k^\lambda\otimes_{B(k)}B(k')$ of $F$-spaces arises from an endomorphism of $E_k^\lambda$. This follows from the description of such endomorphisms given by a proposition in \cite[Section IV.3]{Demazure}.
\end{proof}

\begin{proposition}\label{CSDConstant2}
  Let $R$ be a discrete valuation ring of characteristic $p$ with perfect residue field $k$ and perfection $R\into \Rper$. Let $\Gscr$ be a completely slope divisible group over $R$. There exists a unique isomorphism $\Gscr_{R^\text{perf}}\cong (\Gscr_k)_{R^\text{perf}}$ which is the identity in the fiber over $k$. In case $\Gscr$ has a single slope, this isomorphism is already defined over $R$.
\end{proposition}
\begin{proof}
  Let $\Rper\into R^\text{psh}$ be a strict henselization of $\Rper$ and $\bar k$ the residue field of $R^\text{psh}$. Note that $k\into \bar k$ is an algebraic closure of $k$. Let $\Rper\into R'$ be a finite etale extension of $\Rper$. Using the fact that the relative Frobenius morphism of $R'$ over $\Rper$ is an isomorphism one sees that $R'$ is again perfect. Hence $R^\text{psh}$ is perfect. Hence by Proposition \ref{CSDConstant1} there exists a $p$-divisible group $\Gscr'$ over $\bar k$ such that $\Gscr_{R^\text{psh}}\cong \Gscr'_{R^\text{psh}}$. By taking the special fiber of this isomorphism, we get $\Gscr_{\bar k}\cong \Gscr'$, so that we may take $\Gscr'=\Gscr_{\bar k}$. Then Proposition \ref{ConstantHom} implies that there exists a unique isomorphism $\psi\colon (\Gscr_{\bar k})_{R^\text{psh}}\cong \Gscr_{R^\text{psh}}$ which is the identity in the special fiber. For any $\sigma\in \Aut(R^\text{psh}/\Rper)$, the conjugate of $\psi$ by $\sigma$ is again the identity in the special fiber and thus is equal to $\psi$. Thus $\psi$ is defined over $R^\text{perf}$ by Galois descent. 

In case $\Gscr$ has a single slope, one does not need to pass to $\Rper$ to split the slope filtration. Thus with the same argument as above one obtains $\psi$ over a strict henselization of $R$ and sees that it is defined over $R$ by Galois descent.
\end{proof}

 Now let $R$ be a discrete valuation ring as in Section \ref{sec:FormalSchemes}. The results in this subsection are formulated for $p$-divisible groups over $\Spec(R)$, however below we will work with $p$-divisible groups over $\Spf(R)$. Thus we will need the following:
\begin{proposition}[{\cite[Lemma 4.16]{MessingCrystals}}] \label{BTFormal}
  Let $R'$ be the valuation ring of a complete overfield $K'\subset \hat{\bar K}$ of $K$. The base change functor $\Gscr\mapsto \Gscr_{\Spf(R')}$ from the category of $p$-divisible groups over $\Spec(R')$ to the category of $p$-divisible groups over $\Spf(R')$ is an equivalence.
\end{proposition}
Accordingly we define:
\begin{definition}
  Let $R'$ be the valuation ring of a complete overfield $K'\subset \hat{\bar K}$ of $K$. A $p$-divisible group over $\Spf(R')$ is completely slope divisible if and only if the corresponding group over $\Spec(R')$ is completely slope divisible.
\end{definition}
\subsection{Nice Semiabelian Schemes}\label{sec:NiceSetup}
Let $K$ be a local field of characteristic $p>0$ with valuation ring $R$. Let $\bar K$ be an algebraic closure of $K$ and $\bar R$ the valuation ring of $\bar K$. Let $\Rper\subset \bar R$ be the perfection of $R$. Denote by $k$ (resp. $\bar k$) the residue field of $R$ (resp. $\bar R$).
\begin{definition}
 We call a semiabelian scheme $\CA$ over $\Spec(R)$ \emph{nice} if $\CA$ is an extension of an abelian scheme by a torus over $\Spec(R)$ and the $p$-divisible group $\hat\CA$ over $\Spf(R)$ is completely slope divisible.
\end{definition}

Let $\CA$ be a nice semiabelian scheme over $R$.

\begin{lemma} \label{NiceSubgroup}
  Let $\CB\subset \CA$ be a semiabelian subgroup scheme. Then $\CB$ and $\CA/\CB$ are nice.
\end{lemma}
\begin{proof}
  The ranks of the toric parts of $\CB$ and $\CA/\CB$ are constant since they have constant sum and can only go up upon specialization. Thus by \cite[Corollary 2.11]{ChaiFaltings} both $\CB$ and $\CA/\CB$ are extensions of an abelian scheme by a torus over $S$. Proposition \ref{CSDSubgroup} shows that the formal completions of both group schemes are again completely slope divisible.
\end{proof}

\begin{construction} \label{FConstr}
  We construct an isogeny $F_{\hat \CA}\colon \hat\CA_{\Spf(\Rper)}\to\hat\CA_{\Spf(\Rper)}$ as follows: Propositions \ref{CSDConstant2} and \ref{BTFormal} yield a unique isomorphism $(\hat\CA_k)_{\Spf(\Rper)}\cong \hat \CA_{\Spf(\Rper)}$ which is the identity on the special fiber. The $p$-divisible group $\hat\CA_k$, being defined over the finite field $k$, has a Frobenius endomorphism with respect to $k$. Transfering the base change of this Frobenius endomorphism to $\Spf(\Rper)$ to an endomorphism of $\hat \CA_{\Spf(\Rper)}$ via the above isomorphism yields $F_{\hat \CA}$.
\end{construction}

The following summarizes the relevant properties of $F_{\hat \CA}$:
\begin{proposition} \label{FProperties}
  \begin{enumerate}[(i)]
  \item There exists a good polynomial which annihilates $F_{\hat \CA}$.
  \item The endomorphism $F_{\hat \CA}$ is the Frobenius endomorphism with respect to a suitable model of $\hat\CA_{\Spf(\Rper)}$ over the finite field $k$.
  \item Let $\CB$ be another nice semiabelian scheme over $R$. For any homomorphism $f\colon \CA\to \CB$, the induced homomorphism $\hat f_{\bar R}\colon \hat\CA_{\Spf(\Rper)}\to\hat\CB_{\Spf(\Rper)}$ satisfies $F_{\hat\CB}\circ\hat f_{\Spf(\Rper)}=\hat f_{\Spf(\Rper)}\circ F_{\hat\CA}$. 
  \item In case $\CA$ has a model over $k$, the endomorphism $F_{\hat\CA}$ is the one induced by the Frobenius endomorphism of such a model.
  \item If one replaces $R$ by a finite extension $R'$ contained in $\bar R$, then $F_{\hat\CA}$ is replaced by $F_{\hat\CA}^N$, where $N$ is the degree of the extension of the residue fields of $R$ and $R'$. 
  \end{enumerate}
\end{proposition}
\begin{proof}
  $(i)$ By the construction of $F_{\hat\CA}$ it suffices to show that there exists a good polynomial which annihilates the Frobenius endomorphism of $\CA_k$. This follows from the Riemann hypothesis for abelian varieties, see for example \cite[Fact 3.1]{GhiocaMoosa}. 
 
$(ii)$ This follows directly from the construction.

$(iii)$ Pick isomorphisms $\hat\CA_\Rper\cong (\hat\CA_k)_\Rper$ and $\hat\CB_\Rper\cong (\hat\CB_k)_\Rper$ as in Construction \ref{FConstr}. Under these identifications by Proposition \ref{ConstantHom} the homomorphism $\hat f_\Rper\colon \hat\CA_\Rper\to\hat\CB_\Rper$ arises by base change from its special fiber $\hat f_k\colon\hat\CA_k\to \hat\CB_k$. Since the latter is defined over $k$, it is compatible with the Frobenius endomorphisms of $\hat\CB_k$ and $\hat\CA_k$. This implies $F_{\hat\CB}\circ\hat f_\Rper=\hat f_\Rper\circ F_{\hat\CA}$. 

$(iv)$ This follows directly from the construction.

$(v)$ This follows directly from the construction.

\end{proof}

\begin{theorem} \label{SpecialFChar}
   Let $X\subset \CA_K$ be an irreducible subvariety and $\CX$ its schematic closure in $\CA$. Then the following are equivalent:
  \begin{enumerate}[(i)]
  \item The subvariety $X_{\bar K}$ is special in $\CA_{\bar K}$.
  \item There exist a finite field extension $\tilde K\subset \bar K$ of $K$ with valuation ring $\tilde R$, $x\in X(\tilde K)$ and $n\geq 1$ such that $F^n_{\hat\CA}((\widehat{\CX_{\tilde R}-x})_{\Spf(\tilde{R}^\text{per})})\subset (\widehat{\CX_{\tilde R}-x})_{\Spf(\tilde{R}^\text{per})}$.
  \item There exist $x\in X(\bar K)$, $n \geq 1$ and subset $T\subset \hat\CA(\bar R)\cap (X-x)(\bar K)$ which is Zariski dense in $X_{\bar K}-x$ and which satisfies $F^n_{\hat\CA}(T)\subset T$.
  \end{enumerate}
\end{theorem}
\begin{proof}
  $(i)\implies (ii)$: Using Proposition \ref{FProperties} (iii) we see that if $(ii)$ holds for $X/\Stab_{\CA_{K}}(X)\subset \CA_{K}/\Stab_{\CA_{K}}(X)$, then it holds for $X\subset \CA_{\bar K}$. Hence we may assume that $\Stab(X)=0$. Then there exist a semiabelian variety $B$ defined over a finite field $k'$ containing $k$, a subvariety $Y$ of $B$, a homomorphism $h\colon B_{\bar K}\to \CA_{\bar K}$ with finite kernel and $a\in \CA(\bar K)$ such that $X_{\bar K}=h(Y_{\bar K})+a$. Note that it suffices to prove $(ii)$ after replacing $K$ by a finite field extension contained in $\bar K$. After doing so we may assume $a\in A(K)$ and $k=k'$. After suitably translating we may assume that $0\in Y(k')$. Then $a\in X(K)$, so that after translating $X$ by $-a$ we may assume $a=0$. 

  As $Y$ is defined over $k'$, the set $Y(\bar k)$ is Zariski dense in $Y$, and hence the set $h(Y(\bar k))\subset X(\bar K)$ is Zariski dense in $X_{\bar K}$. Thus there exists $y\in Y(\bar k)$ such that $\CX$ is smooth at $h(x)$ and such that $h$ is flat at $y$. After possibly replacing $k'$ and $K$ by finite field extensions we may assume that $y\in Y(k')$ and $h(y)\in X(K)$. Then after replacing $Y$ by $Y-y$ and $X$ by $X-h(x)$ we may assume that $\CX$ is smooth over $R$ at $0$. By \cite[Proposition 2.7]{ChaiFaltings} the homomorphism $h$ extends to a homomorphism $h\colon B_{\bar R}\to \CA_{\bar R}$. After replacing $K$ by a finite field extension contained in $\bar K$ we may assume that $h$ is defined over $R$. By Proposition \ref{HatXProps} the formal scheme $\hat\CX$ is isomorphic to $\Spf(R[[x_1,\hdots,x_n]])$ for some $n$ and hence is irreducible. Thus it follows from Proposition \ref{CompletionImage} that $\hat\CX$ is the formal schematic image of $\hat Y_{\Spf(R)}$ under $\hat h$.

 We have $F_{\hat\CA}\circ\hat h=\hat h\circ F_{\hat B_R}$ by Proposition \ref{FProperties}. This together with the fact that $\hat Y_{\Spf(R)}$, being defined over $k'$, is invariant under a suitable power of $F_{\hat B_R}$ implies that $\hat\CX=\hat h(\hat Y_{\Spf(R)})$ is invariant under a suitable power of $F_{\hat\CA}$. This shows $(ii)$. 

$(ii)\implies (iii)$: By Proposition \ref{FormalPointsDense} set $T\defeq \hat\CA(\bar R)\cap (X_{\bar K}-x)(\bar K)$ is Zariski dense in $X_{\bar K}-x$. 

$(iii)\implies (i)$: Using Proposition \ref{FProperties} (i), Theorem \ref{InfSpecialCrit} applied to $G=\hat\CA(\bar R)$ and $\Phi=\F^n_{\hat\CA}$ shows that $X_{\bar K}-x$ is special in $\CA_{\bar K}$. 
\end{proof}

\subsection{Choice of a nice valuation} \label{sec:NiceValuation}
\begin{lemma}\label{Blub}
  Let $L$ be an algebraically closed field equipped with a non-archimedean valuation $v$ with valuation ring $R$. Let $\CA$ be a semiabelian scheme over $R$ which is an extension of an abelian scheme over $R$ by a torus over $R$. For $a\in A(L)$, if $na\in\CA(R)$ for some $n\in\BZ^{\not= 0}$, then $a\in\CA(R)$.  
\end{lemma}
\begin{proof}
  Let 
  \begin{equation*}
    \xymatrix{
      0\ar[r] &  \CT\ar[r] & \CA\ar[r]^\pi & \CB \ar[r]&  0
}
  \end{equation*}
 be exact with $\CT$ a torus over $R$ and $\CB$ an abelian scheme over $R$. Then $\pi(a)\in \CB(L)=\CB(R)$. 

Since $R$ is strictly Henselian the flat cohomology group $\operatorname{H}^1(R,\CT)$ is zero. Thus the point $\pi(a)\in\CB(R)$ lifts to a point $a'\in \CA(R)$. Since $R$ is strictly Henselian, the torus $\CT$ is split. Thus for $t=a-a'\in \CT(L)\cong (L^*)^r$ we have $t^n\in (R^*)^r$ which implies that $t\in (R^*)^r$. Hence $a$ lies in $\CA(R)$.
\end{proof}
\begin{proposition}  \label{NiceCompletion}
  Let $L_0$ be a field which is finitely generated over $\BF_p$ and let $A$ be a semiabelian variety over $L_0$. There exists an embedding of $L_0$ into a local field $K$ such that the semiabelian variety $A_K$ extends to a semiabelian scheme $\CA$ over the valuation ring $R$ which is isogenous to a nice semiabelian scheme over $R$.

If we are given a finitely generated subgroup $\Gamma\subset A(L_0)$ (resp. a finite rank subgroup $\Gamma'\subset A(L_0^\per)$) we can pick $v$ such that $\Gamma\subset \CA(R)$ (resp. such that $\Gamma'\subset \CA(\Rper)$, where $\Rper$ denotes the valuation ring of the unique extension of $v$ to $L_0^\per$).
\end{proposition}

\begin{proof}
  There exists a  ring $R_0\subset L_0$ with quotient field $L_0$ which is finitely generated over $\BF_p$ such that $A$ extends to a semiabelian scheme $\CA$ over $R_0$. Since the Newton polygon of $\hat\CA$ is generically constant, after localizing $R_0$ we may assume that $\hat\CA$ has constant Newton polygon. Since the rank of the toric part of $\CA$ is generically constant, we can also assume that this rank is constant. Then $\CA$ is globally an extension of an abelian scheme by a torus by \cite[Corollary 2.11]{ChaiFaltings}. 

If we are given a finitely generated subgroup $\Gamma$ as above, after further localization we may assume that a finite generating set of $\Gamma$, and thus all of $\Gamma$, is contained in $\CA(R_{0})$.
 
If we are given a finite rank subgroup $\Gamma'$ as above, pick a finitely generated subgroup $\Gamma$ of $\Gamma'$ such that $\Gamma'$ consists of divison points of $\Gamma$. Then as before after further localization we may assume that $\Gamma\subset \CA(R_0)$.

By \cite[Lemma 3.1]{ScanlonManinMumford} there exists an embedding $R_0\into R\defeq \BF_q[[t]]$ for a suitable power $q$ of $p$. We pick such an embedding and let $K\defeq \BF_q((t))$. By Theorem \ref{CSDIsogeny}, the $p$-divisible group $\hat\CA_{\Spf(R)}$ is isogenous to a completely slope divisible group $\Gscr$ over $R$. Since an isogeny $\hat\CA_{\Spf(R)}\to\Gscr$ is given by the quotient by a finite group scheme, this shows that we can find a nice semiabelian scheme $\CA'$ over $R$ which is isogenous to $\CA_R$. In case we are given $\Gamma'$ as above, Lemma \ref{Blub} ensures $\Gamma'\subset \CA(\Rper)$. Thus the embedding has the required properties.
\end{proof}

\section{Proof of Mordell-Lang for finitely generated groups}
\label{sec:proof1}

\subsection{A formal Mordell-Lang theorem}
\label{FormalMLsec}
Let $R$ be the valuation ring of a local field $K$ of characteristic $p>0$ and let $\Fm$ be its maximal ideal. Let $\bar R$ be the valuation ring of an algebraic closure $\bar K$ of $K$ and let $\bar \Fm$ be the maximal ideal of $\bar R$. Let $\Gscr$ be a formal group over $\Spf(R)$ which as a formal scheme is isomorphic to $\Spf(R[[x_1,\hdots,x_n]])$. From an isomorphism $\Gscr\cong \Spf(R[[x_1,\hdots, x_n]])$ one gets a bijection $\Gscr(\bar R)\cong \bar \Fm^{\oplus n}$ as in Remark \ref{XBarExplicit}. This endows $\Gscr(\bar R)$ with a valuation topology which is independent of the chosen isomorphism. The fact that $[p]\colon \Gscr\to\Gscr$ acts by zero on the tangent space of $\Gscr$ implies that for all $g\in\Gscr(\bar R)$ the sequence $(p^ng)_{n\geq 0}$ converges to zero with respect to the valuation topology. This implies that the $\BZ$-module structure on $\Gscr(\bar R)$ can be uniquely extended to $\BZ_p$-module structure which is continuous with respect to the valuation topology.

First we prove the following Mordell-Lang statement for formal schemes in positive characteristic:

\begin{theorem} \label{FormalML}
  Let $K$ be a local field of characteristic $p$ with valuation ring $R$ and residue field $k$. Let $\Gscr$ be a formal group over $k$ which as a formal scheme is isomorphic to $\Spf(k[[x_1,\hdots,x_n]])$. Let $\bar \Gamma\subset \Gscr(\Rper)$ be a finite rank $\BZ_p$-module. Let $\Xscr\subset \Gscr_{\Spf(\Rper)}$ be a closed formal subscheme. 

If $\Xscr(\Rper)\cap \bar\Gamma$ is formal-schematically dense in $\Xscr$, then there exist a finite field extension $K'$ of $K$ with valuation ring $R'$, closed formal subschemes $\Xscr_1,\hdots,\Xscr_m$ of $\Xscr_{\Spf(R')}$ and elements $\gamma_1,\hdots,\gamma_m \in\bar\Gamma$ such that $\Xscr_j+\gamma_j$ is defined over the residue field of $R'$ and such that $\Xscr_{\Spf((R')^\text{per})}=\cup_j (\Xscr_j)_{\Spf((R')^\text{per})}$.
\end{theorem}

For the proof of Theorem \ref{FormalML} we need the following lemma:
\begin{lemma} \label{pMult}
  Let $K$ be a local field of characteristic $p$ with valuation ring $R$ and residue field $k$ with $q$ elements. For $i\geq 0$ let $K^{q^i}$ be the field consisting of $q^i$-th powers of elements of $K$ and $R^{q^i}$ the valuation ring of $K^{q^i}$.  Let $\Gscr$ formal group scheme over $k$. Then $q^i\Gscr(R)\subset \Gscr(R^{q^i})$.
\end{lemma}
\begin{proof}
  Let $F\colon \Gscr \to \Gscr$ be the Frobenius endomorphism of $\Gscr$ with respect to $k$. By \cite[Section VII A.4]{SGA3I} there exists a ``Verschiebung'' endormorphism $V\colon \Gscr\to\Gscr$ such that $[q]= F\circ V=V\circ F$. Using the fact that $F^i(G(R))\subset G(R^{q^i})$ this implies the claim.
\end{proof}

\begin{proof}[Proof of Theorem \ref{FormalML}]

Since $\bar\Gamma$ is a finite rank $\BZ_p$-module, after possibly replacing $K$ by a finite extension we may assume that $\bar\Gamma\subset \Gscr(R)$. Since then $\Xscr(\Rper)\cap\bar\Gamma \subset \Gscr(R)$, Proposition \ref{BaseChangeClosure} implies that $\Xscr$ is defined over $R$. Let $\Xscr_1,\hdots, \Xscr_m$ be the irreducible components of $\Xscr$. Then $\Xscr_i(\bar R)\cap\bar\Gamma$ is formal-schematically dense in $\Xscr_i$ for each $i$. It suffices to prove the claim for each $\Xscr_i$, that is we may assume that $\Xscr$ is irreducible.

 Since the group $\bar\Gamma/p^i\bar\Gamma$ is finite for all $i\geq 0$ and since $\Xscr$ is irreducible, it follows using Lemma \ref{IrrUnion} that we can choose $(\gamma_i)_{i\geq 0}\in \bar\Gamma^{\BZ^{\geq 0}}$ such that $p^i\bar\Gamma\cap (\Xscr+\gamma_i)(R)$ is formal-schematically dense in  $\Xscr+\gamma_i$ and such that $\gamma_{i+1}\equiv \gamma_i \pmod{p^i\bar\Gamma}$ for all $i\geq 0$. Since the finite rank $\BZ_p$-module $\bar\Gamma$ is complete for the $p$-adic topology, there exists $\gamma\in\bar\Gamma$ such that $\gamma\equiv\gamma_i \pmod{p^i\bar\Gamma}$ for $i\geq 0$. Then $\Xscr+\gamma$ is the formal schematic closure of $p^i\bar\Gamma$ for all $i\geq 0$. Let $q$ be the number of elements of $k$. For $i\geq 0$ let $K^{q^i}$ be the field consisting of $q^i$-th powers of elements of $K$ and $R^{q^i}$ the valuation ring of $K^{q^i}$.  Since $q^i\bar\Gamma\subset \Gscr(R^{q^i})$ by Lemma \ref{pMult}, Proposition \ref{BaseChangeClosure} implies that $\Xscr+\gamma$ is defined over $R^{q^i}$. Thus Lemma \ref{FormalFieldDefinition} implies that $\Xscr+\gamma$ is defined over $k$.
\end{proof}

\subsection{Proof of Mordell-Lang for finitely generated groups}
Using Theorem \ref{FormalML} we can now give an algebraic proof of Theorem \ref{IntroML}.
\begin{theorem}[Hrushovski] \label{pfgML}
  Let $L$ be an algebraically closed field of positive characteristic. Let $A$ be a semiabelian variety over $L$, let $X\subset A$ an irreducible subvariety and $\Gamma\subset A(L)$ a finitely generated subgroup. If $X(L)\cap \Gamma$ is Zariski dense in $X$, then $X$ is a special subvariety of $A$. 
\end{theorem}
\begin{proof}
  Let $L_0\subset L$ be a field which is finitely generated over $\BF_p$ such that $A$ arises by base change from an abelian variety $A_0$ over $L_0$, such that $X$ arises by base change from a subvariety $X_0$ defined over $L_0$ and such that $\Gamma\subset A_0(L_0)$. By Proposition \ref{NiceCompletion} there exists an embedding $L_0\into K$ into a local field $K$ and a semiabelian scheme $\CA$ over $R$ which has generic fiber $A_K$, is isogenous to a nice semiabelian scheme over $R$ and satisfies $\Gamma\subset \CA(R)$. Let $\bar K$ an algebraic closure of $K$ and denote $\bar R$ the valuation ring of $\bar K$, by $k$ the finite residue field of $R$ and by $\CX$ the schematic closure of $X_{0,K}$ inside $\CA$. 

By Lemma \ref{SpecialBaseChange} it suffices to prove that $\CX_{\bar K}$ is special in $\CA_{\bar K}$. Using Lemma \ref{ReplaceIsogenous} we can replace $\CA$ by an isogenous semiabelian variety, so that we may assume that $\CA$ is nice.

 Since we have an exact sequence
  \begin{equation*}
    0 \to \hat\CA(R)\to \CA(R)\to \CA(k) \to 0
  \end{equation*}
with $\CA(k)$ finite, after replacing $\CX$ by a suitable translate we may assume that $(\Gamma\cap\hat\CA(R))\cap \CX(R)$ is schematically dense in $\CX$. After replacing $\Gamma$ by $\Gamma\cap\hat\CA(R)$ we may thus assume $\Gamma\subset\hat\CA(R)$. 

Let $F_{\hat\CA}\colon \hat\CA_{\Spf(\Rper)}\to\hat\CA_{\Spf(\Rper)}$ be the endomorphism given by Construction \ref{FConstr}. By Proposition \ref{FProperties} (ii) there exist a $p$-divisible group $\Gscr$ over $k$ and an isomorphism $\Psi\colon \Gscr_{\Spf(\Rper)}\cong \hat\CA_{\Spf(\Rper)}$ under which $F_{\hat\CA}$ corresponds to the Frobenius endomorphism of $\Gscr$ with respect to $k$. 

Let $\bar\Gamma$ be the closure of $\Gamma$ with respect to the valuation topology on $\hat\CA(\bar R)$. Since $\Gamma$ is finitely generated this is a finite rank $\BZ_p$-module. Let $\Xscr\subset \hat\CA$ be the formal schematic closure of $\Gamma\cap\CX(R)$ inside $\hat\CA$. Then $\Xscr$ is the formal schematic closure of $\Xscr(R) \cap\bar\Gamma$ in $\hat\CA$. Thus by Corollary \ref{BaseChangeSchematicClosure} the formal scheme $\Xscr_{\Spf(\Rper)}$ is the schematic closure of $\Xscr(R) \cap\bar\Gamma$ in $\hat\CA_{\Spf(\Rper)}$. Thus by Theorem \ref{FormalML} applied to $\Psi(\Xscr_{\Spf(\Rper)})$ and $\Psi(\bar\Gamma)$, there exist closed formal subschemes  $\Xscr_1,\hdots,\Xscr_m$ of $\Xscr_{\Rper}$, elements $\gamma_1,\hdots,\gamma_m \in\bar\Gamma$ and $n\geq 0$ such that $F^n_{\hat\CA}(\Xscr_j+\gamma_j)\subset \Xscr_j+\gamma_j$ and $\Xscr=\cup_j \Xscr_j$. Since $\Gamma\cap \CX(R)\subset \Xscr(\bar R)$, the set $\Xscr(\bar R)$ is schematically dense in $\CX_{\bar R}$. Lemma \ref{XBarRUnion} implies $\Xscr(\bar R)=\cup_i\Xscr_i(\bar R)$. Since $\CX_{\bar R}$ is irreducible there exists $i$ such that $\Xscr_i(\bar R)$ is schematically dense in $\CX_{\bar R}$. Thus Proposition \ref{SpecialFChar} (iii) holds for $T\defeq \Xscr_i(\bar R)+\gamma_i$. By Proposition \ref{SpecialFChar} this implies that $\CX_{\bar K}$ is special in $\CA_{\bar K}$. 
\end{proof}

\section{Towards full Mordell-Lang}
The full Mordell-Lang conjecture in positive characteristic is the following conjecture:
\begin{conjecture} \label{fullML}
   Let $L$ be an algebraically closed field of positive characteristic. Let $A$ be a semiabelian variety over $L$, let $X\subset A$ be an irreducible subvariety and let $\Gamma\subset A(L)$ be a subgroup of finite rank. If $X(L)\cap \Gamma$ is Zariski dense in $X$, then $X$ is a special subvariety of $A$. 
\end{conjecture} 
In this section we show that in case $A$ is an ordinary or supersingular abelian variety by combining our method with a reduction due to Ghioca, Moosa and Scanlon, Conjecture \ref{fullML} can be reduced to the following special case:
\begin{conjecture}\label{fullMLreduction2}
  Let $L_0$ be a field which is finitely generated over $\BF_p$, let $L$ an algebraic closure of $L_0$ and let $L_0^\per$ the perfect closure of $L_0$ in $L$. Let $A$ be a semiabelian variety over $L_0$ and $X \subset A_{L_0^\per}$ an irreducible subvariety. Assume that the canonical morphism $\Tr_{L/\bar \BF_p}A\to A$ is defined over $L_0$, that there exists a finite subfield $\BF_q$ of $L_0$ over which $\Tr_{L/\bar\BF_p}A$ can be defined and that $\Stab_{A_{L_0^\text{per}}}(X)$ is finite. If $X(L_0^\per)$ is Zariski dense in $X_{0}$, then a translate of $X$ by an element of $A(L_0^\text{per})$ is defined over $L_0$.
\end{conjecture}
\begin{remark}
  We expect that Conjecture \ref{fullMLreduction2} holds without the condition on the field of definition of the morphism $\Tr_{L/\bar \BF_p}A\to A$, the existence of $\BF_q$ as above and the condition on $\Stab_{A_{L_0^\text{per}}}(X)$. However we were unable to deduce Conjecture \ref{fullMLreduction2} from Conjecture \ref{fullML} without these conditions. Note that they can always be achieved by dividing by $\Stab_{A_{L_0^\text{per}}}(X)$ and by replacing $L_0$ by a suitable finite extension.
\end{remark}
\begin{remark}
  Consider the case of a supersingular abelian variety $A$. Then $A$ is isogenous to a power of a supersingular elliptic curve (c.f. \cite[Theorem 4.2]{Oort1}), and thus in order to prove Conjecture \ref{fullML} for $A$, it suffices to prove it for powers of supersingular elliptic curves. But these can be defined over a finite field and for such abelian varieties, Conjecture \ref{fullML} is proven in \cite{GhiocaMoosa}. Thus the results in this section are only interesting in case $A$ is ordinary.
\end{remark}
\begin{lemma} \label{PerfectionPoints}
  Let $L_0$ be a finitely generated field of characteristic $p>0$ with perfection $L_0\into L_0^\text{per}$ and let $A$ be a semiabelian variety over $L_0$. 
  \begin{itemize}
  \item [(i)] The group $A(L_0^\text{per})$ has finite rank.
  \item [(ii)] Let $\Gamma \subset A(L_0^\text{per})$ be a subgroup and $n\geq 0$. The group $\Gamma/p^n\Gamma$ is finite.
  \end{itemize}
\end{lemma}
\begin{proof}
  $(i)$ For $n\geq 0$ the group $p^nA(L_0^{p^{-n}})$ is contained in $A(L_0)$. This together with the fact that $A(L_0)$ is finitely generated implies that $A(L_0^\text{per})$ has finite rank. 

$(ii)$ By Claim 1 in the proof of Theorem 2.2 of \cite{GhiocaMoosa} the torsion subgroup of $A(L_0^\text{per})$ is finite. Hence both the rank and the size of the torsion subgroup of the finitely generated groups $\Gamma_i\defeq \Gamma\cap A(L_0^{p^{-i}})$ are bounded as $i\geq 0$ varies. Thus the size of the groups $\Gamma_i/p^n\Gamma_i$ is bounded as $i$ varies. Let $\bar\Gamma_i$ be the image of $\Gamma_i/p^n\Gamma_i$ in $\Gamma/p^n\Gamma$. The groups $\bar\Gamma_i$ form an ascending sequence of finite groups of bounded size, thus for all $i\gg 0$ the $\bar\Gamma_i$ coincide. Since $\Gamma/p^n\Gamma$ is the union of the $\bar\Gamma_i$ this shows $(ii)$.
\end{proof}

\begin{proposition} \label{Reduction2}
  Conjecture \ref{fullML} implies Conjecture \ref{fullMLreduction2}.
\end{proposition}
\begin{proof}
  After translating $X$ by an element of $X(L_0^\text{per})$ we may assume that $0\in X$. By Lemma \ref{PerfectionPoints} the group $A(L_0^\text{per})$ has finite rank. Hence by Conjecture \ref{fullML} the subvariety $X_L$ is special in $A_L$. Thus by Lemma \ref{SpecialEquivalence} there exist a subvariety $Y\subset \Tr_{L/\bar\BF_p}A$ over $\bar\BF_p$ and $a\in A(L)$ such that $X_L=\tau(Y_L)+\Stab_A(X)_L+a$, where $\tau$ is the natural homomorphism $\Tr_{L/\bar\BF_p}A\to A$ which by assumption is defined over $L_0$. Since by assumption $\Stab_A(X)$ is finite and $X$ is irreducible, there exists $a'\in A(L)$ such that $\tau(Y_L)=X_L+a'$.

 The fact that $0\in X$ implies $a'\in \Img(\tau)(L)$. Thus $X\subset \Img(\tau)$. After replacing $A$ by $\Img(\tau)$ we may assume that $\tau$ is an isogeny. We fix a model of $\Tr_{L/\bar\BF_p}A$ over a finite field $\BF_q$ contained in $L_0$ and let $F\colon \Tr_{L/\bar\BF_p}A\to \Tr_{L/\bar\BF_p}A$ the Frobenius endomorphism of $\Tr_{L/\bar\BF_p}A$ with respect to $\BF_q$. Since by \cite[Theorem 6.12]{ConradTrace} the homomorphism $\tau$ is purely inseparable, the induced map $\tau\colon \Tr_{L/\bar\BF_p}A(L_0^\text{per})\to A(L_0^\text{per})$ is surjective. Let $Y'$ be an irreducible component of the Zariski closure of $\tau^{-1}(X(L_0^\text{per}))\subset \Tr_{L/\bar\BF_p}A(L_0^\text{per})$. Let $\Gamma\defeq \Tr_{L/\bar\BF_p}A(L_0)$. This is a finitely generated $F$-invariant subgroup of $\Tr_{L/\bar\BF_p}A(L)$ and $Y'(L)\cap (\cup_{n\geq 0}F^{-n}\Gamma)=Y'(L)\cap \Tr_{L/\bar\BF_p}A(L_0^\text{per})$ is Zariski dense in $Y'$. Thus by \cite[Proposition 3.9]{GhiocaMoosa} for some $n\geq 0$ the set $Y'(L)\cap F^{-n}\Gamma\subset Y'(L_0^{q^{-n}})$ is Zariski dense in $Y'$. Thus $X(L_0^{q^{-n}})$ is Zariski dense in $X$. Since the group $A(L_0^{q^{-n}})/q^nA(L_0^{q^{-n}})$ is finite, there thus exists $a''\in A(L_0^{q^{-n}})$ such that $(a''+q^nA(L_0^{q^{-n}}))\cap X(L_0^{q^{-n}})$ is Zariski dense in $X$. Since $q^nA(L_0^{q^{-n}})\subset A(L_0)$ and since Zariski closure commutes with base change this implies that $X-a''$ is defined over $L_0$.
\end{proof}

\subsection{A specialness criterion}
Let $K$ be a local field of characteristic $p>0$ with valuation ring $R$ and residue field $k$. Let $\bar K$ be an algebraic closure of $K$ and $\bar R$ the valuation ring of $\bar K$. Let $\CA$ be an abelian scheme over $R$ such that the Newton polygon of $\hat\CA$ is constant. Denote $\CA_K$ by $A$.

The following definition is somewhat ad hoc and adapted to our present needs:
\begin{definition}\label{RelSpecialDef}
 Let $i\geq 0$ and $X$ a subvariety of $A$. Denote the schematic closure of $X$ in $\CA$ by $\CX$. We say that $X$ is \emph{$K/K^{p^i}$-special} in $A$ if $X$ is irreducible and there exists an abelian variety $B$ over $K^{p^i}$, a subvariety $Y$ of $B$ over $K^{p^i}$, a homomorphism $h\colon B_{K}\to A/\Stab_A(X)$ with finite kernel and an element $a\in \widehat{(\CA/\Stab_\CA(\CX)}(\bar R)$ such that $(X/\Stab_A(X))_{\bar K}=h(Y_{\bar K})+a$.
\end{definition}
\begin{theorem} \label{RelSpecial}
  Assume that $A$ is ordinary or supersingular. Let $X$ be an irreducible subvariety of $A$ containing $0$. If $X$ is $K/K^{p^i}$-special in $A$ for all $i\geq 0$, then $X_{\bar K}$ is special in $A_{\bar K}$. 
\end{theorem}
\begin{proof}

Since the Newton polygon of $\hat\CA$ is constant, by Proposition \ref{CSDIsogeny} the $p$-divisible group $\hat\CA$ is isogenous to a completely slope divisible $p$-divisible group. Since isogenies of $p$-divisible groups are quotients by finite subgroup schemes, this implies that there exists a nice abelian scheme $\CA'$ over $R$ together with an isogeny $f\colon \CA\to \CA'$. It follows directly from Definition \ref{RelSpecialDef} that $f(X)$ is again $K/K^{p^i}$-special in $\CA'_K$ for all $i\geq 0$. By Lemma \ref{SpecialImage}, the subvariety $X_{\bar K}$ of $A_{\bar K}$ is special if and only if $f(X_{\bar K})$ is special in $\CA'_{\bar K}$. Thus we may replace $\CA$ by $\CA'$ and can assume that $\CA$ is nice.

  Let $\CX$ be the schematic closure of $X$ in $\CA$. After dividing by $\Stab_\CA(\CX)$ we may assume that $\Stab_A(X)=0$. By Proposition \ref{GeomIrrComp}, after replacing $K$ by a finite field extension we may assume that each irreducible component of $\hat\CX$ is geometrically irreducible. 

Let $i\geq 0$. By assumption there exists a abelian variety $B$ over $K^{p^i}$, a subvariety $Y\subset B$, a homomorphism $h\colon B_K\to A$ and $a\in\hat \CA(\bar R)$ such that $X_{\bar K}=h(Y_{\bar K})+a$. Let $\CB$ be the N\'eron model of $B$ over $R^{p^i}$. Since $\CA$ is the N\'eron model of $A_K$, the homomorphism $h$ extends to a homomorphism $\tilde h\colon \CB_R\to \CA$. Since the generic fiber of the completion $\hat{\tilde h}\colon \hat\CB_R\to \hat\CA$ is an isogeny, Tate's conjecture implies that $\hat{\tilde h}$ is an isogeny. Since the Newton polygon of $\hat\CA$ is constant, the existence of this isogeny implies that the Newton polygon of $\hat\CB$ is constant. Thus, by Theorem \ref{CSDIsogeny} and Proposition \ref{BTFormal}, there exists a completely slope divisible $p$-divisible group $\Gscr$ over $\Spf(R^{p^i})$ together with an isogeny $h'\colon \Gscr\to \hat\CB$. 

Let $\CY$ be the schematic closure of $Y$ in $\CB$. Pick a finite field extension $K'\subset \bar K$ of $K$ with valuation ring $R'$ such that $a\in\hat\CA(R')$. Then $\tilde h(\CY_{R'})+a=\CX_{R'}$. Let $\Yscr\defeq (h')^{-1}(\hat\CY)\subset \Gscr$. Then $\Xscr_i\defeq \hat{\tilde h}(h'(\Yscr_{\Spf(R')}))+a\subset \hat\CX_{\Spf(R')}$, where $\hat{\tilde h}(h'(\Yscr_{\Spf(R')}))$ is the formal schematic image of $\Yscr_{\Spf(R')}$. By Lemma \ref{PreimageImage} we have $h'(\Yscr_{\Spf(R')})=\hat\CY_{\Spf(R')}$. Thus by Proposition \ref{CompletionImage} each irreducible component of $\Xscr_i$ is an irreducible component of $\hat\CX_{\Spf(R')}$. Since each irreducible component of $\hat\CX$ is geometrically irreducible, it follows that $\Xscr_i$ is the union of some of the irreducible components of $\hat\CX$.

Using the fact that $\hat\CA$, and hence $\Gscr$, has a single slope, Propositions \ref{CSDConstant2} and \ref{BTFormal} yield unique isomorphisms $\psi\colon \Gscr\isoto (\Gscr_k)_{\Spf(R^{p^i})}$ and $\psi'\colon \hat\CA\isoto (\hat\CA_k)_{\Spf(R)}$ which are the identity in the special fiber. Under these identifications, by Proposition \ref{ConstantHom} the homomorphism $\hat{\tilde h}\circ h'\colon \Gscr_R \to\hat\CA$ arises by base change from its special fiber. Thus the identity $\Xscr_i=\hat{\tilde h}(h'(\Yscr))+a$ shows that a translate of $\psi'(\Xscr_i)$ by an element of $\hat\CA_k(\bar R)$ is defined over $R^{p^i}$.

 As we saw above, each $\Xscr_i$ is the union of some of the irreducible component of $\Xscr$. Since there are only finitely many such components it follows that there exists $\Xscr\subset \hat\Xscr$ such that $\Xscr=\Xscr_i$ for infinitely many $i$. Theorem \ref{FormalDescent} implies that there exists a finite field extension $K'\subset \bar K$ with valuation ring $R'$ and $x\in \Xscr(R')$ such that $\psi'(\Xscr_{\Spf(R')}-x)$ is defined over $k$. By construction, the endomorphism $F_{\hat\CA}$ corresponds to the Frobenius endomorphism of $\hat\CA_k$ under $\psi'$. Thus $T\defeq \Xscr(\bar R)-x$ satisfies $F_{\hat\CA}(T)\subset T$. By Proposition \ref{FormalPointsDense} $(ii)$ the set $T$ is Zariski dense in $\CX_{\bar K}-x$. Thus condition $(iii)$ of Theorem \ref{SpecialFChar} is satisfied and Theorem \ref{SpecialFChar} implies that $X_{\bar K}$ is special in $A_{\bar K}$. 
\end{proof}

\subsection{Proof of the reduction}

\begin{theorem}
  Conjecture \ref{fullML} and Conjecture \ref{fullMLreduction2} are equivalent for abelian varieties which are ordinary or supersingular.
\end{theorem}
\begin{proof}
One direction was already proved in Proposition \ref{Reduction2} above. Now we prove that Conjecture \ref{fullMLreduction2} implies Conjecture \ref{fullML}.

By \cite[Theorem 2.2]{GhiocaMoosa}, Conjecture \ref{fullML} is implied by:
\begin{conjecture}\label{fullMLreduction1}
  Let $L_0$ be a field which is finitely generated over $\BF_p$, $L$ an algebraic closure of $L_0$ and $L_0^\per$ the perfect closure of $L_0$ in $L$.  Let $A$ be a semiabelian variety over $L_0$, let $X \subset A$ an irreducible subvariety and let $\Gamma\subset A(L_0^{\per})$ a subgroup of finite rank. If $X(L_0^{\per})\cap \Gamma$ is Zariski dense in $X_L$, then $X_L$ is a special subvariety of $A_L$. 
\end{conjecture}
Furthermore, from the proof of \cite[Theorem 2.2]{GhiocaMoosa}, one sees that if one only wants to prove Conjecture \ref{fullML} for abelian varieties which are ordinary or supersingular, then it suffices to prove Conjecture \ref{fullMLreduction1} for such abelian varieties. This is what we do now. 

After replacing $L_0$ by a finite extension we may assume that the canonical morphism $\tau\colon \Tr_{L/\bar \BF_p}A\to A$ is defined over $L_0$ and that there exists a finite subfield $\BF_q$ of $L_0$ over which $\Tr_{L/\bar\BF_p}A$ can be defined. After dividing by $\Stab_{A}(X)$ we may assume that $\Stab_{A}(X)=0$.

By Proposition \ref{NiceCompletion} there exists an embedding of $L_0$ into a local field $K$ such that the abelian variety $A_{K}$ extends to a abelian scheme $\CA$ over the valuation ring $R$ of $K$ which is isogenous to a nice abelian scheme over $R$ and such that $\Gamma\subset A(L_0^\text{per})\subset \CA(\Rper)$.

Let $\CX$ be the schematic closure of $X_K$ inside $\CA$. By Lemma \ref{SpecialBaseChange} it suffices to prove that $X_{\bar K}$ is special in $\CA_{\bar K}$. 

Since there is an exact sequence 
\begin{equation*}
  0\to \hat\CA(\Rper) \to \CA(\Rper)\to \CA(k) \to 0
\end{equation*}
with $\CA(k)$ finite there exists $\gamma\in \Gamma\cap X(L_0^\text{per})$ such that $X(\bar K)\cap (\gamma+(\Gamma\cap\hat\CA(\Rper))$ is Zariski dense in $X_{\bar K}$. After replacing $L_0$ by a finite extension we may assume that $\gamma\in A(L_0)$. Then after translating $X$ by $X-\gamma$ and replacing $\Gamma$ by $\Gamma\cap\hat\CA(\Rper)$ we may assume that $\Gamma\subset \hat\CA(\Rper)$ and $0\in X$.

Fix a finite subfield $\BF_q$ of $L_0$ over which $\Tr_{L/\bar\BF_p}A_L$ can be defined. By Theorem \ref{RelSpecial} it now suffices to show that $\CX_K$ is $K/K^{q^i}$-special in $A$ for all $i\geq 0$. Thus we fix such an $i$. We will work with the abelian variety $A^{(q^i)}$, which is naturally defined over $L_0^{q^i}$, together with the Verschiebung homomorphism $V\colon A_{L_0}^{(q^i)}\to A$. By Lemma \ref{PerfectionPoints} the group $\Gamma/q^i\Gamma$ is finite. Thus there exists $\gamma\in \Gamma$ such that $X(L_0^\text{per})\cap (\gamma+q^i\Gamma)$ is Zariski dense in $X_{\bar K}$. Let $T\subset \Gamma$ such that $\gamma+q^iT=X(L_0^\text{per})\cap (\gamma+q^i\Gamma)$. Let $F\colon A\to A^{(q^i)}$ be the relative $q^i$-Frobenius and $Y\subset A_{L_0^\text{per}}^{(q^i)}$ an irreducible component of the Zariski closure of $F(T)\subset A^{(q^i)}(L_0^\text{per})$. The fact that $V\circ F=[q^i]$ implies $V(Y)+\gamma=X$.

By \cite[Theorem 6.4]{ConradTrace} the formation of $\Tr_{L/\bar\BF_p}A_{L}$ commutes with purely inseparable base change. Thus $\tau^{(q^i)}\colon \Tr_{L/\bar\BF_p}A_{L}\cong (\Tr_{L/\bar\BF_p}A_{L})^{(q^i)}\to A_{L}^{(q^i)}$ is the $L/\bar\BF_q$-trace of $A_{L}^{(q^i)}$. This morphism is defined over $L_0^{q^i}$. Since $\Stab_{A}(X)=0$ the stabilizer $\Stab_{A_{L_0^{\text{per}}}^{(q^i)}}(Y)$ is contained in the kernel of $V$ and thus is finite. Thus Conjecture \ref{fullMLreduction2} applied to to $Y\subset A_{L_0^\text{per}}^{(q^i)}$ gives an element $a\in A^{(q^i)}(L_0^\text{per})$ such that $Y+a$ is defined over $L_0^{q^i}$.

 We consider $a$ as an element of $A^{(q^i)}(K^\text{per})$. By the choice of embedding $L_0\into K$, the element $V(a)\in A(L_0^\text{per})$ lies in $\CA(\Rper)$. Hence $q^ia=F(V(a))\in\CA^{(q^i)}(\Rper)$. Thus $a\in \CA^{(q^i)}(\Rper)$ by Lemma \ref{Blub}. Since the natural map $\CA^{(q^i)}(R^{q^i})\to \CA^{(q^i)}(k)$ is surjective, there exists $a'\in \CA^{(q^i)}(R^{q^i})$ such that $a-a'\in \hat\CA^{(q^i)}(\Rper)$. Then $Y+a-a'$ is defined over $K^{q^i}$ and $V(Y+a-a')+V(a'-a)+\gamma=X$ with $V(a'-a)+\gamma\in \hat\CA(\Rper)$. Thus $\CX_{\bar K}$ is $K/K^{q^i}$-special and we are done.
\end{proof}


\bibliography{references}
\bibliographystyle{plain}

\end{document}